\documentclass[10pt, a4paper]{amsart}
\usepackage{amsmath,amstext,amscd, amssymb,txfonts, epsfig, psfrag, color, multicol, graphicx, enumitem, array}
\usepackage[normalem]{ulem}
\usepackage[colorlinks=true, pdfstartview=FitV, linkcolor=blue, citecolor=blue, urlcolor=blue]{hyperref}

\theoremstyle{plain}

\newtheorem{thm}{Theorem}[section]
\newtheorem{lemma}[thm]{Lemma}

\newtheorem{cor}[thm]{Corollary}
\newtheorem{prop}[thm]{Proposition}

\theoremstyle{definition}
\newtheorem{defn}[thm]{Definition}

\newtheorem{remark}[thm]{Remark}

\newtheorem{Open questions}[thm]{Open questions}
\newtheorem{Open question}[thm]{Open question}
\newtheorem{Open problems}[thm]{Open problems}
\newtheorem{Open problem}[thm]{Open problem}

\definecolor{magenta}{rgb}{.5,0,.5} 
\definecolor{dred}{rgb}{.5,0,0} 
\definecolor{green}{rgb}{0,.5,0} 
\definecolor{blue}{rgb}{0,0,0.5} 
\definecolor{black}{rgb}{0,0,0} 
\definecolor{vdgreen}{rgb}{0,.3,0} 
\definecolor{vdred}{rgb}{.3,0,0} 
\definecolor{red}{rgb}{1,0,0} 
\definecolor{orange}{rgb}{1.00,0.50,0}

\def\Bbb{\mathbb}
\def\bar{\overline}

\def\R{\Bbb{R}}
\def\Z{\Bbb{Z}}
\def\N{\Bbb{N}}

\renewcommand{\N}{\Bbb{N}}

\def\ni{\noindent}

\def\Sol{\hbox{\rm Sol}}

\def\F+L{\hbox{$\textup{F}\!_+\textup{L}$}}

\def\ssm{\smallsetminus}

\renewcommand{\H}{\mathcal{H}}

\newcommand{\0}{\parbox{10pt}{0}}

\def\ms{\medskip}

\def\onto{{\kern3pt\to\kern-8pt\to\kern3pt}}

\def\<{\langle}
\def\>{\rangle}
\def\|{{\ |\ }}

\def\z{\mathbb{Z}}

 \def\e{\mathbf{e}}

\newcommand{\set}[1]{\left\{#1\right\}}

\newcommand{\abs}[1]{\left|#1\right|}

\renewcommand{\ni}{\noindent}

\renewcommand{\ms}{\medskip}

\def\*{^{\star}}

\newcommand{\SB}[2]{\mbox{\footnotesize{$\left(\!\!\!\begin{array}{c}{#1} \\
{#2}\end{array}\!\!\!\right)$}}}

\newcommand{\T}[3]{\mbox{${}_{#2} {}^{#1} {}_{#3}$}}

\begin{document}

\title[{Lamplighters, metabelian groups, and horocyclic products of trees}
]{Lamplighters, metabelian groups, \\ and horocyclic products of trees }

\author{Margarita Amchislavska and Timothy Riley}

\date \today

\begin{abstract}
\ni  
Bartholdi, Neuhauser and Woess proved that a family of metabelian groups including lamplighters  have a striking geometric manifestation as $1$-skeleta of  horocyclic products of trees.   The purpose of this article is to give an elementary account of this result, to widen the family addressed to include the infinite valence case (for instance $\Z \wr \Z$), and to make the translation between the algebraic and geometric descriptions explicit.  

In the  rank-2 case, where the groups concerned include a celebrated example of Baumslag and Remeslennikov, we give the translation by means of a combinatorial `lamplighter description'.  This elucidates our proof in the general case which proceeds by manipulating polynomials.

Additionally, we show that the  Cayley 2-complex of a suitable presentation of Baumslag and Remeslennikov's example is  a horocyclic product of three trees. 

 \ms

\footnotesize{\ni \textbf{2010 Mathematics Subject
Classification:  20F05, 20F16, 20F65}  \\ \ni \emph{Key words and phrases:} horocyclic product, lamplighter, metabelian}
\end{abstract}

\maketitle

 \section{Introduction}

Our conventions throughout will be $[a,b] = a^{-1}b^{-1} ab$ and $a^{nb} = b a^n b^{-1}$ for group elements $a$, $b$ and integers $n$.  Our group actions are on the right. 

\subsection{The original lamplighter group $(\Z/2\Z)\wr\Z$}

Denote \mbox{$(\Z/2\Z)\wr\Z$} by $\Gamma_1(2)$. As an abelian group the ring $(\Z/2\Z)[x,x^{-1}]$ is isomorphic to the additive group $\bigoplus_{i \in \Z} (\Z/2\Z)$ of finitely supported sequences of zeros and ones.  By definition  \mbox{$\Gamma_1(2) = \bigoplus_{i \in \Z} (\Z/2\Z)\rtimes \Z$}, and so can also  be expressed as $(\Z/2\Z)[x,x^{-1}]\rtimes \Z$, and this provides a convenient description of the action of the $\z$-factor, namely a generator of  the   $\Z$-factor  acts on $(\Z/2\Z)[x,x^{-1}]$ by multiplication by $x$. 

Elements  $\left( \sum_{j \in \Z}  f_{j} x^{j}, k \right) \in \Gamma_1(2)$ can be visualized as a street (the real line) with lamps at all integer locations, a lamplighter  located by lamp $k$, and, for each $f_j=1$, the lamp at $j$ is lit. We will call this the lamplighter model for $\Gamma_1(2)$. The identity element $(0,0)$ corresponds to all lights being turned off and the lamplighter at location $0$. Figure~\ref{element of gamma_1(2)} illustrates $(x^{-4}+1+x+x^3,5)\in\Gamma_1(2)$.

\begin{figure}[ht]
\centering
\includegraphics[scale=0.6]{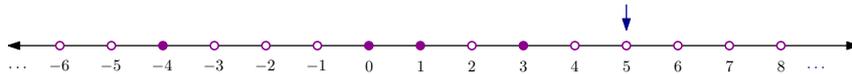}
\caption{An element $(x^{-4}+1+x+x^3,5)$ of $\Gamma_1(2)$. The lamps at positions $-4,0,1$ and $3$ are turned on and the lamplighter is standing by the lamp at location $5$.}
\label{element of gamma_1(2)}
\end{figure}

As we will show in Section~\ref{presentations section},
$$\left\langle  \ a, t \ \left| \ a^2=1,   \left[a,a^{t^k}\right]=1 \ (k \in \Z)  \  \right. \right\rangle$$
is a presentation for $\Gamma_1(2)$. 

Elements of $\Gamma_1(2)$ expressed as words on $a$ and $t$ can be visualized on the lamplighter model above by starting with the model for the identity element, reading off one letter at a time from left to right: upon reading $t$ we move the lamplighter one unit to the right (hence upon reading $t^{-1}$ we move one unit to the left), and upon reading $a^{\pm 1}$ we flip the switch on the lamp at which the lamplighter is currently located. For example, both $t^{-4}at^{4}atat^2at^2$ and $at^{-1}at^4at^{-7}at^3at^{2}at^4$ express the element pictured in Figure~\ref{element of gamma_1(2)}.

\subsection{Cayley graphs}

The \emph{Cayley graph} of a group $G$ with respect to a generating set $A$ is the graph which has elements of $G$ as its vertex set and, for every $g \in G$ and $a \in A$, has a directed edge   labeled $a$ from $g$ to $ga$.  
The presentation complex of a finitely presented group $G=\langle A ~|~R\rangle$ denoted by $P_G$ is a 2-dimensional cell complex which has a single vertex, one loop at the vertex for each generator of $G$, and one $2$-cell for each relation in the presentation glued along the corresponding edge-loop. The universal cover $\widetilde{P_G}$ of $P_G$ is called the \emph{Cayley $2$-complex} of $G$, and the $1$-skeleton of $\widetilde{P_G}$ gives the Cayley graph of $G$ with respect to this presentation. 

A group acts \emph{geometrically} on a metric space if the action is cocompact, by isometries, and properly discontinuous (that is, every two points have neighbourhoods such that only finitely many group elements translate one neighbourhood in such a way that it intersects the other).  For example, the action of a group $G$ on itself by left-multiplication naturally extends to such an action on a Cayley graph that is defined using a finite generating set.

\subsection{A primer on horocyclic products of trees}

Part of the infinite binary tree $\mathcal{T}_{\z/2\z}$ with every vertex having valence $3$ and equipped with a height function $h$ is shown in Figure~\ref{H_1(2)}. A horocyclic product is constructed from two copies of $\mathcal{T}_{\z/2\z}$ by taking the subset 
$$ \left. \mathcal{H}_1(\z/2\z):= \set{ \, (p_0, p_1) \in \mathcal{T}_{\z/2\z}\times \mathcal{T}_{\z/2\z} \, \right| \,  h(p_0)+h(p_1) =0 \, }$$
of $\mathcal{T}_{\z/2\z} \times \mathcal{T}_{\z/2\z}$. 
In Section~\ref{Horocyclic product defn section} we will give precise definitions and will generalize this construction to products of $n+1$ trees by taking  the subset of $(n+1)$-tuples of points in the tree whose  heights sum  to zero.  

\begin{figure}[ht!]
\centering
\includegraphics[scale=0.85]{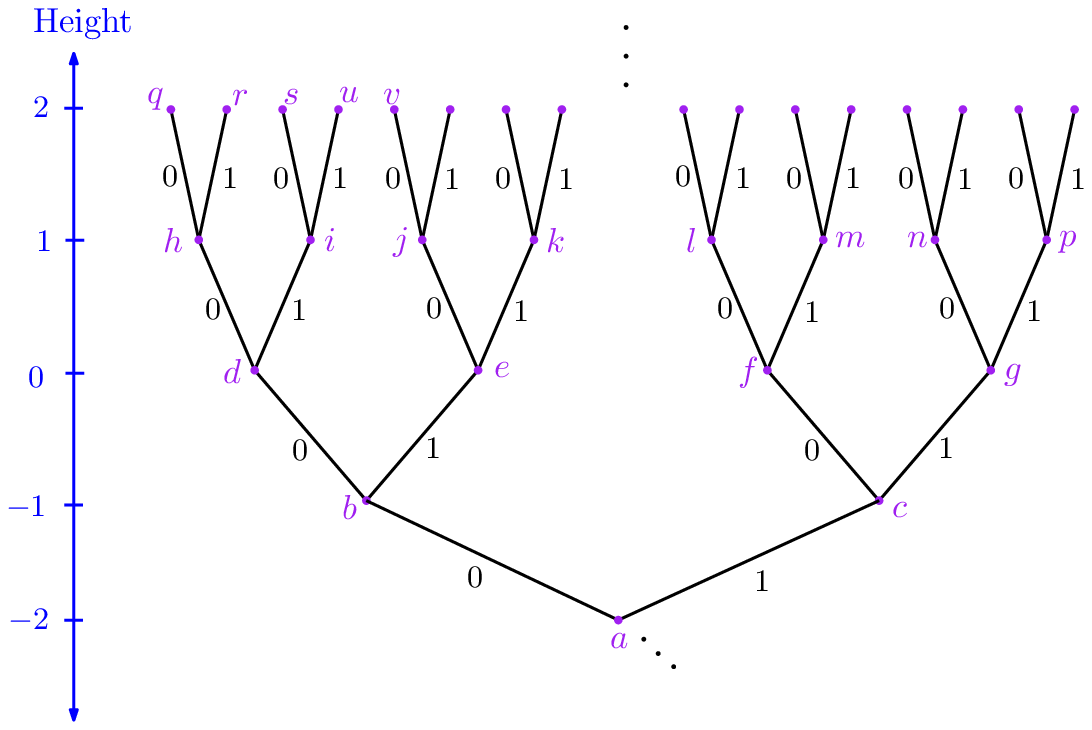}
\caption[A part of an infinite binary tree with a height function.]{A part of an infinite binary tree with a height function. Vertices in $\mathcal{H}_1(\z/2\z)$ include $(d,d)$, $(d,e)$, $(d,f)$, $(e,d)$, $(j,c)$, $(v,a)$, and $(a,u)$. Edges in $\mathcal{H}_1(\z/2\z)$ include $\big\{(d,e), (h,b)\big\}$, $\big\{(d,e), (b,k)\big\}$, and $\big\{(u,a), (i,c)\big\}$.}
\label{H_1(2)}
\end{figure}

The starting point for this article is that this striking generic construction turns out to give a Cayley graph of $\Gamma_1(2)$ --
\begin{prop} \label{starting prop}
The Cayley graph of $\Gamma_1(2)$ with respect to the generating set $\{ a, at \}$ is $\mathcal{H}_1(\z/2\z)$.
\end{prop}

This result originates with P.~Neumann and R.~M\"oller in 2000.   They noticed that, with respect to a suitable generating set, the Cayley graph of $\Gamma_1(2)=(\Z / 2 \Z) \wr \Z$ is a \emph{highly-arc-transitive digraph} constructed by M\"oller  in \cite{Moller}, which is the horocyclic product $\mathcal{H}_1(\z/2\z)$ of two infinite binary trees \cite{MN}.  See also \cite{BaW, BW, Woess}.

Proposition~\ref{starting prop} is a special case (with $n=1$ and $R=\z/2\z$) of Theorem~\ref{main}   which will identify Cayley graphs of generalized lamplighter groups with the $1$-skeleta of horocyclic products of trees.  A mild generalization (allowing other rings in place of $\Z/2\Z$)   is proved in Section~\ref{rank 1}.

\subsection{Generalized lamplighter groups.}\label{general lamp groups}

Another group we can consider is $\Z \wr \Z$ which we denote by $\Gamma_1$. Again, as an abelian group the ring $\Z[x,x^{-1}]$ is isomorphic to the additive group $\bigoplus_{i \in \Z} \Z$ of $\Z$-indexed finitely supported sequences of integers. So $\Gamma_1$ can also  be expressed as $\Z[x,x^{-1}]\rtimes \Z$ where  a generator of  the   $\Z$-factor  acts on $\Z[x,x^{-1}]$ by multiplication by $x$. The model for $\Gamma_1$ is similar to that of $\Gamma_1(2)$, except each lamp has $\z$-worth of brightness levels. A presentation for $\Gamma_1$ is $\left\langle  \ a, t \ \left| \  \left[a,a^{t^k}\right]=1 \ (k \in \Z)  \  \right. \right\rangle$, which is similar to that of $\Gamma_1(2)$ except that $a$ has infinite order.

Similarly, for any commutative ring with unity $R$, we can construct a group $\Gamma_1(R)=R[x,x^{-1}]\rtimes \Z$ and consider the model where the lamps have $|R|$-worth of brightness levels. In this notation, $\Gamma_1(2)=\Gamma_1(\z/2\z)$ and $\Gamma_1=\Gamma_1(\z)$. The case where $n=1$ of Theorem~\ref{main} states that the horocyclic product of two $R$-branching trees $\mathcal{H}_1(R)$ (defined in Section~\ref{tree defn section}) is the Cayley graph of $\Gamma_1(R)$ with respect to a suitable generating set (proved in Section~\ref{rank 1}).

We can generalize these constructions further. The group $\Gamma_2$ is a celebrated example of Baumslag  \cite{Baumslag} and Remeslennikov \cite{Rem}
$$\Z \left[ x, x^{-1}, (1+x)^{-1}\right] \rtimes \Z^{2}$$ where, if the $\Z^2$-factor is $\langle t,s \rangle$, the  actions of $t$ and $s$ are by  multiplication by $x$ and $(1+x)$, respectively.   It was the first example of a finitely presented group with an abelian normal subgroup of infinite rank --- specifically, the derived subgroup $\left[ \Gamma_2, \Gamma_2\right]$. We will show in Proposition~\ref{the presentations for gamma 2} that one of the presentations for $\Gamma_2$ is $$ \left\langle \ a,s,t \ \left| \   [a,a^t]=1, \ [s,t]=1, \ a^s=aa^t \ \right. \right\rangle.$$
An analogous lamplighter model for general $\Gamma_2(R)=R \left[ x, x^{-1}, (1+x)^{-1}\right] \rtimes \Z^{2}$ will be discussed in Section~\ref{Gamma_2 lamplighter}. Restricting to the case where $n=2$, Theorem~\ref{main} states that the $1$-skeleton of the horocyclic product of three $R$-branching trees $\mathcal{H}_2(R)$ is the Cayley graph of $\Gamma_2(R)$ with respect to a suitable generating set (proved in Section~\ref{rank 2}).

We can generalize these constructions even further to obtain the family of groups $\Gamma_n(R)$ that figure in Theorem~\ref{main} defined as follows.


Suppose $R$ is any  commutative ring with unity.   
 
For   $n=1,2, \ldots$,  let 
$A_n(R)$ be the polynomial ring 
 $$R \left[ x, x^{-1}, (1+x)^{-1}, \ldots, (n -1+x)^{-1} \right].$$
 For  $\mathbf{h} = (h_0, \ldots, h_{n-1}) \in \Z^n$ and  $f \in A_n(R)$, define   
$$f \cdot \mathbf{h} \ := \  f x^{h_0} (1+x)^{h_1} \cdots (n-1+x)^{h_{n-1}}.$$ 
 Then
 $\Gamma_n(R)   := \ A_n(R) \rtimes \Z^n$ where   
the group operation is   $(f,\mathbf{h})(\hat{f},\hat{\mathbf{h}})   =   (f + \hat{f} \cdot {\mathbf{h}},\mathbf{h} + \hat{\mathbf{h}})$.   

 This definition can be conveniently repackaged as:     $$\Gamma_n(R) \ \cong \ \set{  \ \left. \left(\begin{array}{cc}x^{k_0}(1+x)^{k_1}\cdots (n-1+x)^{k_{n-1}}    & f \\0 & 1\end{array}\right) \  \right| \  k_0,  \ldots,  k_{n-1} \in \Z, \ f \in A_n(R) \  },$$
where the matrix multiplication naturally realizes the semi-direct product structure of the group.

 For brevity, define $\Gamma_n := \Gamma_n(\Z)$ and $\Gamma_n(m) := \Gamma_n(\Z/m\Z)$. 
 
It will prove natural for us to index the coordinates of $\Z^n$   by $0, \ldots, n-1$.  Accordingly, we use   $\e_0, \ldots, \e_{n-1}$ to denote the standard basis for $\Z^n$.


In higher rank, the examples originate with Baumslag, Dyer, and Stammbach in \cite{BD, BaS}.   Bartholdi, Neuhauser and Woess \cite{BNW} studied the family including  $\Gamma_{n}(m)$ for $n=1, 2, \ldots$ and $m \in \N$ such that $2, 3, \ldots, n-1$ are invertible in $\Z/m\Z$. And recently, Kropholler and Mullaney \cite{KM}, building on Groves and Kochloukova~\cite{GK}, studied $\Gamma_n(\Z[1/(n-1)!]) \rtimes \Z$ where a generator of the   $\Z$-factor  acts as multiplication by $(n-1)!$ on the $A_n(\Z[1/(n-1)!])$-factor in $\Gamma_n(\Z[1/(n-1)!])$ and  trivially on the $\Z^n$-factor.  To put it another way,  these groups are $A_n(\Z[1/(n-1)!]) \rtimes \Z^{n+1}$, defined like $\Gamma_n(\Z[1/(n-1)!])$, but with a generator of the additional $\Z$-factor acting on  $A_n(\Z[1/(n-1)!])$ by   multiplication by $(n-1)!$.  

 \subsection{Cayley graphs of generalized lamplighter groups.}

The main theorem we address in this article is:
  
\begin{thm} \label{main}
For $n =1, 2, \ldots$, if $2, \ldots, n-1$ are invertible in $R$, then  the 1-skeleton of $\mathcal{H}_n(R)$ is the Cayley graph  of $\Gamma_n(R)$ with respect to the generating set $$\set{ \left. \, (r, \e_j), \ (r, \e_j)(r, \e_k)^{-1} \  \right| \  r \in R, \  0 \leq  j, k  \leq  n-1  \textup{ and }   j < k \, }.$$   In particular, if $\abs{R} < \infty$, then  $\Gamma_n(R)$ acts geometrically on $\mathcal{H}_n(R)$.  
\end{thm}

For $R$ finite,  this theorem is due to Bartholdi, Neuhauser \& Woess \cite{BNW}.  (Instead of   working with $A_n(R)$ and insisting that  $2, \ldots, n-1$ are invertible in $R$, they work more generally with polynomials $R[x, ( \ell_0 +x)^{-1}, \ldots, (\ell_{n-1} +x)^{-1}]$ such that the pairwise differences $\ell_i - \ell_j$ are all invertible.  Our treatment could be extended to this generality if desired.)  We aim here to give as elementary, explicit and transparent a proof as possible for general  $\Gamma_{n}(R)$.  The proof in \cite{BNW}  proceeds via manipulations of formal Laurent series.  We will work with `lamplighter models' as far as possible---the cases $n=1$ and $n=2$---and use these models to illuminate a proof in the general case which involves suitably manipulating  polynomials.

Theorem~\ref{main} fits into a broader context which can be found in the introduction to \cite{Margulis} (as we thank C.~Pittet for pointing out).  In the case where $R$ is the field $\mathbb{F}_p$, the trees arise  from valuations on $\mathbb{F}_p(\!(x)\!)$ (cf.\ Section~4.2 of \cite{GK}), and this leads to 
$\Gamma_2(\mathbb{F}_p)$ being a cocompact lattice in $\Sol_5(\mathbb{F}_p(\!(x)\!))$ (Proposition~3.4 of \cite{CT}), and we  presume generalizes to  $\Gamma_n(\mathbb{F}_p)$ in $\Sol_{2n+1}(\mathbb{F}_p(\!(x)\!))$.   This provides the formalism adopted by  Bartholdi, Neuhauser \& Woess in \cite{BNW} in their proofs. However our perspective is that the theorem relates two elementary (and starkly different) objects: horocyclic products of trees and  a family of metabelian groups defined using polynomial rings, and there should be a  proof  which is intrinsic to those concepts and is correspondingly elementary.  We aim here to provide such a proof to clarify the relationship and explore how far the ideas can be pushed. 

The $n=1$ and  $n=2$ cases of the theorem motivate us to give (in Section~\ref{presentations section})  some group presentations which reflect the horocyclic product structure.  One such presentation  then features in this embellishment of an $n=2$ case of Theorem~\ref{main}:

\begin{thm} \label{Cayley complex} 
 $\mathcal{H}_2(\Z)$  is the Cayley 2-complex with respect to this presentation  of $\Gamma_2$: 
$$\left\langle \ \lambda_i, \mu_i, \nu_i \  (i \in \Z)   \ \left| \   \lambda_i = \nu_i \mu_i, \  \lambda_{i+j} = \mu_i\nu_j  \  (i,j \in \Z)  \ \right. \right\rangle.$$ 
\end{thm}

\subsection{The organization of this article}

In Section~\ref{significance} we explain the significance of the family $\Gamma_n(R)$.  They have   compelling  applications and properties   and  other manifestations and they bear comparison with other important families such as Bieri--Stallings groups, the Lie groups $\textup{Sol}_{2n+1}$, and Baumslag--Solitar groups.   
In Section~\ref{Horocyclic product defn section} we define the trees $\mathcal{T}_R$ and their horocyclic products $\mathcal{H}_n(R)$, and explain some of their features.      
We prove Theorem~\ref{main} in the case $n=1$ in Section~\ref{rank 1}.  This introduces some of the key ideas in a straight-forward setting.  In Section~\ref{rank 2}, we give a proof for the $n=2$ case which contains most of the ideas of the general proof, but  we are able to present them in purely combinatorial terms using a \emph{lamplighter  description} of  $\Gamma_{2}(R)$.  We  explain our proof for general $n$ in Section~\ref{rank n}.    In Section~\ref{presentations section} we   discuss presentations for   $\Gamma_{n}(R)$ and then we  prove   Theorem~\ref{Cayley complex} in Section~\ref{Cayley complex section}.

 \subsection{Acknowledgements}  We have  enjoyed discussions about these groups with a number of people, especially, Sean~Cleary, David~Fisher, Martin~Kassabov, Peter~Kropholler, Yash~Lodha, Jon~McCammond,  Christophe~Pittet, Melanie~Stein, Jennifer~Taback, and Peter~Wong.  We also thank  John~Baez and Allen~Knutson for recognizing that the cell structures considered in Section~\ref{cell-structure section} are assembled from hypersimplices.

\section{The significance of the family $\Gamma_n(R)$} \label{significance}

Here are some of the applications, properties, and cousins of the groups $\Gamma_n(R)$. 

Instances of the family $\Gamma_n(R)$ and the related horocyclic products have featured in some major breakthroughs.  Baumslag and Remeslennikov's construction of $\Gamma_2$ precipitated their theorem that every finitely generated metabelian group embeds  in a finitely presented metabelian group \cite{Baumslag5, Rem}.  

Grigorchuk, Linnell, Schick, and {\.Z}uk    showed that the $L^2$-Betti numbers of  Riemannian manifold with torsion-free fundamental group need not be integers (answering a strong version of a question of Atiyah  \cite{Atiyah}) by  constructing a $7$-dimensional such manifold with fundamental group $\Gamma_{2}(2)$ and third $L^2$-Betti number $1/3$ in \cite{GLSZ}. 
  
Diestel and Leader  in \cite{DL} put forward the  horocyclic product of an infinite 2-branching and an infinite 3-branching tree  as a candidate to answer a question of Woess as to whether there is a vertex-transitive graph  not quasi-isometric to a Cayley graph.   Eskin, Fisher and Whyte~\cite{EFW} verified this.  (Accordingly, the 1-skeleta of $\mathcal{H}_{n}(\Z/m\Z)$ of Section~\ref{horo section} are termed \emph{Diestel--Leader graphs} in \cite{BNW}.)  Woess recently wrote an account of this breakthrough and its history  \cite{Woess2}.  
  
Eskin, Fisher and Whyte  \cite{EFW} also classified lamplighter groups up to quasi-isometry. Dymarz~\cite{Dymarz} used lamplighter examples to show that quasi-isometric finitely generated groups need not be bilipshitz equivalent.  In both cases, the horocyclic product view-point  was essential to their analyses. 

A number of properties of these groups have been identified.   

Bartholdi \& Woess \cite{BaW} studied the asymptotic behaviour of the $N$-step return probabilities of a simple random walk on a horocyclic product of two regular (finitely) branching trees.
  Woess~\cite{Woess}  described  positive harmonic functions in terms of the boundaries of the two trees. 
 Bartholdi, Neuhauser \& Woess~\cite{BNW} identified the $\ell^2$-spectrum of the simple random walk operator and studied the Poisson boundary    for  a large class of group-invariant random walks on   horocyclic products of trees.

A group $G$ is  \emph{of type $\mathcal{F}_n$} if there exists a $K(G,1)$  (an Eilenberg--Maclane space---a CW-complex whose fundamental group is $G$ and which has contractible universal cover) with finite $n$-skeleton. All groups are $\mathcal{F}_0$, being finitely generated is equivalent to  $\mathcal{F}_1$, and being finitely presentable is equivalent to $\mathcal{F}_2$.
Bartholdi, Neuhauser \& Woess \cite{BNW} show that $\H_{n}(\Z/m\Z)$ is $(n-1)$-connected but not $n$-connected and deduce that $\Gamma_{n}(m)$  is  of type $\mathcal{F}_n$ but not of type $\mathcal{F}_{n+1}$ when $1, \ldots, n-1$  are invertible in $\Z/m\Z$.  Kropholler \& Mullaney \cite{KM} use Bieri--Neumann--Strebel invariants to prove that $\Gamma_n  (\Z[1/(n-1)!])   \rtimes \Z$ (as defined in Section~\ref{general lamp groups}) is  of type $\mathcal{F}_n$ but not of type $\mathcal{F}_{n+1}$. The Bieri--Stallings groups \cite{Bieri, Stallings} exhibit the same finiteness properties, and bear close comparison with the family  $\Gamma_{n}(2)$  in that both are  level sets in products of trees (just the height functions concerned differ). 

Cleary \& Taback   \cite{CT1} showed that, with respect to a standard generating set, $\Gamma_{1}(2)$ has \emph{unbounded dead-end depth}: there is no $L >0$ such that for every group element $g$, there is a  group element further from the identity than $g$ that is within a distance less than $L$ from $g$.   (Cf.\  Question 8.4 in   \cite{Bestvina},   which  Erschler observed can be resolved using $\Gamma_1(2)$.)    Cleary \& Riley \cite{CRwithcorrection} exhibited $\Gamma_{2}(2)$ as the first finitely presentable group known to have the same property.   By finding  a combinatorial formula for the word metric, Stein \& Taback \cite{ST}  showed that, with respect to generating sets for which the Cayley graphs are horocylcic products,  $\Gamma_{n}(m)$ have no regular language of geodesics and have unbounded dead-end depth.    We understand that Cleary has unpublished work and Davids \& Taback have work in progress on whether or not almost convexity holds for $\Gamma_2(2)$ with respect to certain generating sets.

De Cornulier \& Tessera showed that the  Dehn function of $\Gamma_{2}(2)$ grows quadratically \cite{CT}, and  Kassabov \& Riley \cite{KR2} that that of    $\Gamma_2$ grows exponentially.   
 
The horocyclic product construction can be applied to any family of spaces  with height functions.    A fruitful alternative to  $\mathcal{T}_{\Z/m\Z}$ is  the hyperbolic plane $\mathbb{H}^2$, viewed as the upper half of the complex plane, with height function given by $\log_q ( \textup{Im} \, z)$   for some fixed $q >1$.  The horocyclic product of $n$ copies of  $\mathbb{H}^2$ (each with the same $q>1$) is a manifold $\textup{Sol}_{2n-1}$.  (Varying $q$ is a dilation.) The horocyclic product of   $\mathcal{T}_{\Z/p\Z}$  and  $\mathbb{H}^2$ with parameter $q$  is termed \emph{treebolic space} in \cite{BSSW}.  When $p=q$ it is shown to be a model space for the Baumslag--Solitar group $\langle a,b \mid b^{-1}ab=a^p \rangle$---that is, the group acts on the space cocompactly  by  isometries.

These constructions and their parallels have been pursued particularly by Woess and his coauthors \cite{BNW, BaW,  BSSW,   BSW, BrW, Woess},  focusing   on stochastic processes, harmonic maps,  and  boundaries.  He gives an introduction in \cite{Woess2}.   Additionally, the boundaries of these various horocyclic products admit similar analyses, which is  why the work of Eskin, Fisher \& Whyte \cite{EF, EFW, EFW0, EFW2} encompasses both $\Sol_3$ and lamplighter groups. 
Dymarz~\cite{Dymarz2013} also exploits the parallels. 

The parallel is promoted to absolute agreement when one passes to asymptotic cones.  After all, the asymptotic cones of $\mathcal{T}_{\Z/m\Z}$  for $m \geq 2$ and of $\mathbb{H}^2$ are both the everywhere $2^{\aleph_0}$-branching $\R$-tree.  The height functions on $\mathcal{T}_{\Z/m\Z}$ and $\mathbb{H}^2$  induce a height function on this $\R$-tree in such a way that the asymptotic cones of a horocyclic product of $k$ spaces, each of which is either $\mathcal{T}_{\Z/m\Z}$ or $\mathbb{H}^2$, is the horocyclic product of $k$ $\R$-trees.  So, for instance, for $m \geq 2$,   the Baumslag--Solitar groups $\textup{BS}(1,m)$,   $\textup{Sol}_3$, and   $\Gamma_{2}(m)$  all have the same asymptotic cones.   (This observation is essentially in Bestvina~\cite{BestvinaPL}.)
  
Another striking manifestation, set out  in \cite[Remark 4.9]{BNW} (building on the $n=1$ case in  \cite{Nekrashevych}), of most $\Gamma_{n}(m)$ is as automata groups.

 \section{Horocyclic products of trees} \label{Horocyclic product defn section}
 
 \subsection{$R$-branching trees}\label{tree defn section}
 
We let $\mathcal{T}_R$ denote the  \emph{$R$-branching tree}. This is the simplicial tree in which every vertex has   $1+ \abs{R}$ neighbours.    Equip $\mathcal{T}_R$  with the  natural path metric  in which every edge has length one.    Any choice of infinite directed geodesic ray $\rho: \R \to \mathcal{T}_R$ with $\Z \subseteq \R$ mapping to the vertices along the ray  determines a height  (or \emph{Busemann}) function $h : \mathcal{T}_R \to  \R$ by  $$h(p) \ = \  \rho^{-1}(q) + d(p,q)$$ where $q$ is the point on the ray closest to $p$.       Figure~\ref{tree heights} gives some examples of calculations of heights.

 \begin{figure}[ht]
   \psfrag{m}{$-1$}  \psfrag{0}{$0$}   \psfrag{1}{$1$}   \psfrag{2}{$2$}  \psfrag{r}{$\rho$}    \psfrag{h}{$h$}     \psfrag{p}{$p$}   \psfrag{q}{$q$}   \psfrag{P}{$p'$}   \psfrag{Q}{$q'$} 
 \centerline{\epsfig{file=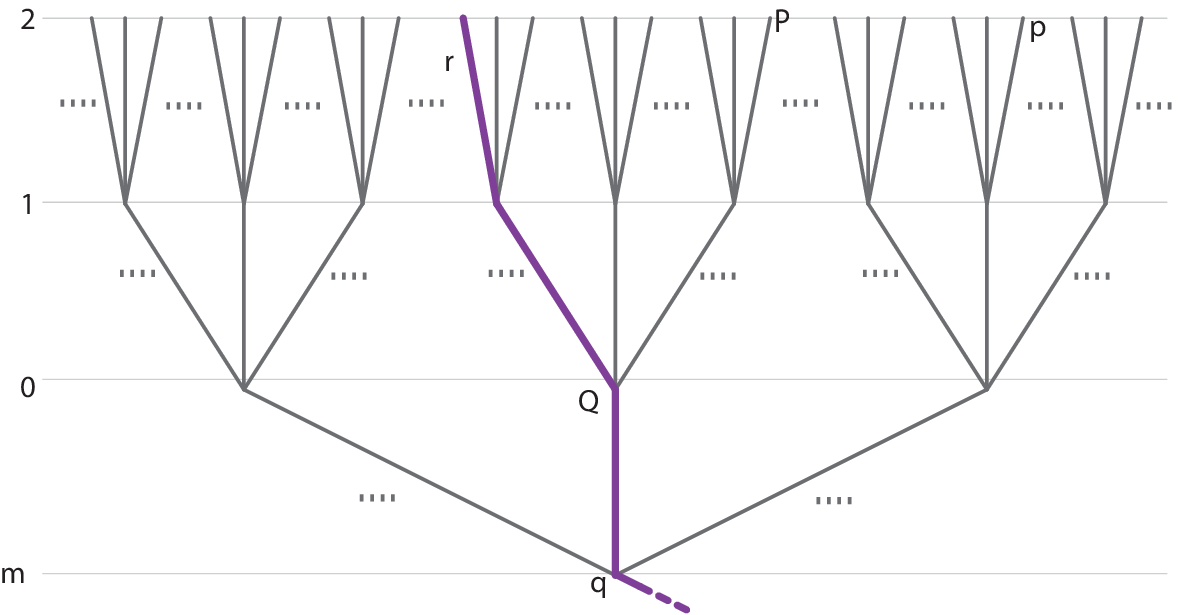}} \caption{The tree $\mathcal{T}_R$ with an infinite geodesic ray  $\rho$ determining a height function $h$.  For example,
$h(p)=\rho^{-1}(q) + d(p,q)=-1+3=2$ and $h(p')=\rho^{-1}(q') + d(p',q')=0+2=2$.} \label{tree heights}
\end{figure}

Label the edges emanating upwards from any given vertex in $\mathcal{T}_R$ by the elements of $R$ in such a way that the edges traversed by $\rho$ are all labeled $0$.   Then we can specify a unique \emph{address} for each vertex in $\mathcal{T}_R$ as follows.  

\begin{lemma}[Addresses of vertices in $\mathcal{T}_R$] \label{addresses}
Vertices $v$ in $\mathcal{T}_R$ are in bijective correspondence with pairs consisting of an integer (the height of $v$) and a finitely supported  sequence of elements of $R$ (the labels on the edges that a  downwards path starting at $v$ follows).  
\end{lemma}

This lemma is easily proved.  The sequences are finitely supported because   the last non-zero entry in the sequence indicates where the downwards path becomes confluent with $\rho$.     
 
\subsection{The horocyclic product of $R$-branching trees}  \label{horo section}
 
 The \emph{horocyclic product} of $n+1$ copies of $\mathcal{T}_R$ is $$\mathcal{H}_n(R) \  := \ \set{ \, (p_0, \ldots, p_n) \in \mathcal{T}_R^{n+1} \, \left| \,  \sum_{i=0}^n h(p_i) =0 \, \right.}.$$    
It is naturally an  $n$-complex:   $(p_0, \ldots, p_n)$ 
is in the $k$-skeleton if and only if  $$\abs{ \rule{0mm}{8pt} \set{ i \mid h(p_i) \in \Z}} \ \geq \  n- k.$$  Equivalently, if we view $\mathcal{T}^{n+1}_R$ as a cubical complex in the natural way, then the $k$-cells of $\mathcal{H}_n(R)$ are the intersections of the $(k+1)$-cells of $\mathcal{T}^{n+1}_R$  with $\mathcal{H}_n(R)$.

Figure~\ref{Tullia's picture} shows a horocyclic product of two 3-branching rooted trees of depth 2, and so a portion of $\mathcal{H}_{1}(\Z/3\Z)$.   Nine upwards- and nine downwards-3-branching trees are apparent in this graph.

 \begin{figure}[ht]
\centerline{\epsfig{file=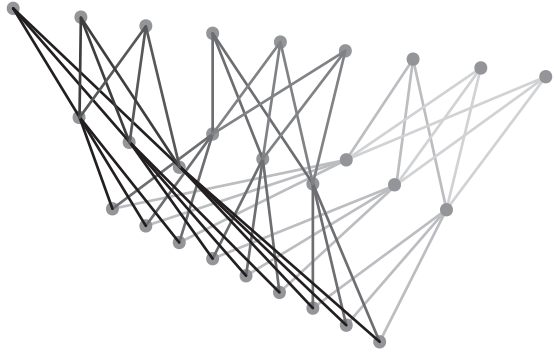}} \caption{A portion of $\mathcal{H}_{1}(\Z/3\Z)$, after a figure by   Dymarz in \cite{Dymarz}.}\label{Tullia's picture}
\end{figure} 
 
\subsection{Cell-structure}  \label{cell-structure section}
 
It will not be required in our proofs of theorems that follow, but we include a description here of the cell-structure of  $\mathcal{H}_n(R)$, which turns out to be attractively exotic and so adds to the lure of family groups $\Gamma_n(R)$.   Some of the details given here  were also identified in Section~4.1 of \cite{BNW}.  
 
To understand the cell-structure of  $\mathcal{H}_n(R)$ it helps to consider   the case $\mathcal{H}_{n}(1)$ where $R$ is the zero ring (with only one element---we do not insist $0 \neq 1$ in a ring), or equivalently  $\Z/1\Z$.   Recall that  $\mathcal{T}_1$   is simply the real line subdivided into unit intervals (known as the \emph{apeirogon}) and  $\mathcal{H}_{n}(1)$ is  the horocyclic product of $n+1$ copies of $\mathcal{T}_1$.    In other words   $\mathcal{H}_{n}(1)$ is the slice through  the standard tessellation of $\R^{n+1}$ by unit $(n+1)$-cubes by the hyperplane  $$H \ := \ \set{ \left.  (x_0, \ldots, x_n) \in \R^{n+1}  \, \right| \,  x_0 + \cdots + x_n =0}.$$

Given that the height-preserving  map  $\mathcal{T}_R \onto \mathcal{T}_1$ that  collapses the branching  induces  a map   $\mathcal{H}_{n}(R)  \onto \mathcal{H}_{n}(1)$, we can view  $\mathcal{H}_{n}(R)$ as many copies  of $\mathcal{H}_{n}(1)$  branching along  $\mathcal{H}_{n-1}(R)$ subcomplexes. 

But what is the cell-structure of $\mathcal{H}_{n}(1)$?  What tessellation  of $\R^n$ (alternatively called a \emph{honeycomb})  does it give?   

The first two examples are readily identified:  $\mathcal{H}_{1}(1)$ is  the apeirogon and  $\mathcal{H}_{2}(1)$ is the tessellation of $\R^2$  by equilateral triangles of side-length $\sqrt{2}$.

The vertices of  $\mathcal{H}_{n}(1)$  are the points where $H$ intersects the 1-skeleton of the tessellation of $\R^{n+1}$ by unit-cubes, in other words the points $(x_0, \ldots, x_n)$ such that $x_0 + \cdots + x_n =0$ and at least $n$ (therefore all) of the coordinates $x_i$ are integers.  So the vertex set of $\mathcal{H}_{n}(1)$ is  $\set{ \left. (x_0, \ldots, x_n) \in \Z^{n+1}  \,  \right| \, x_0 + \cdots + x_n =0 }$, which is known as  the \emph{$A_n$ lattice}.

 The  vectors $\set{ \left.  \e_0 - \e_j  \,  \right| \,  1 \leq j   \leq n}$ generate the parallelepiped 
  \begin{align*}
  P \  = \ \ & \set{ \left.  \sum_{j=1}^n r_j(\e_0-\e_j) \  \right|  \ 0\leq r_j\leq 1 } \\
   \ = \  \ &  \set{ \left.   (x_0, \ldots, x_n) \in \R^{n+1}   \, \right| \,     -1\leq x_1, \ldots, x_n \leq 0 \textup{ and } x_0 + \cdots + x_n =0  }       
 \end{align*}
whose translates $\mathbf{x} + P$, as $\mathbf{x}$ ranges over lattice points, tessellate $H$.   The span of any  $k$ vectors   in  $\set{  \e_0 - \e_j  \mid  1 \leq j   \leq n}$   is a subspace of  $\R^{n+1}$ over which all but $k+1$ coordinates are constantly zero, and so is a subset of the $(k+1)$-skeleton of the tessellation by unit cubes.   So, for every $k$, the $k$-cells of $P$ are a subset of the $k$-skeleton of $\mathcal{H}_{n}(1)$, and $\mathcal{H}_{n}(1)$ is the tessellation formed by the translates of some subdivision of $P$.  This subdivision  is   by \emph{hypersimplices} (also known as  \emph{ambo-simplices}).     

The \emph{$(k,n+1)$-hypersimplex}  (where $k=1,\ldots, n$) is the $n$-dimensional polytope defined in the following three linearly equivalent  ways \cite{Doliwa}.  
\begin{enumerate}
\item The convex hull of the midpoints of the $(k-1)$-cells of the regular $n$-simplex  $\set{ \left. (x_0, \ldots, x_n) \in \R^{n+1} \, \right| \, 0 \leq x_0, \ldots, x_n \leq 1 \textup{ and } x_0 + \cdots + x_n =1  }.$  
\item The convex hull of the $\SB{n+1}{k}$ points  in $\R^{n+1}$ that have $k$ coordinates all $1$ and the remaining $n+1-k$ all $0$.  
\item $\set{ \left. (x_0, \ldots, x_n) \in \R^{n+1} \, \right| \, 0 \leq x_0, \ldots, x_n \leq 1 \textup{ and } x_0 + \cdots + x_n =k  }.$ \label{hyper}
\end{enumerate}
Observe that $P$ is the intersection of $H$ with the union of the cubes $[k-1,k] \times [-1,0]^n$ where $k=1, \ldots, n$.  The intersection of $H$ with   $[k-1,k] \times [-1,0]^n$ is  $$\set{ \left. (x_0, \ldots, x_n) \in \R^{n+1} \, \right| \, k-1 \leq x_0 \leq  k, \ -1 \leq x_1, \ldots, x_n \leq 0  \textup{ and } x_0 + \cdots + x_n =0},$$ which is mapped to the
$(k,n+1)$-hypersimplex as given by (iii) by the linear equivalence $x_0 \mapsto k-x_0$ and  $x_i \mapsto -x_i$ for $i=1, \ldots, n$.  So (see \cite{Doliwa})  $P$ is assembled from  $(k,n+1)$-hypersimplices, one for each   $k = 1, \ldots, n$.        
For instance, in the case of $\mathcal{H}_{3}(1)$, the parallelepiped $P$  is assembled by attaching tetrahedra (a $(1,4)$- and a $(3,4)$-hypersimplex) to a pair of opposite faces of an octahedron (a $(2,4)$-hypersimplex).

This is the same cellular structure that is obtained from the $A_n$  lattice in $\R^n$ by taking the \emph{Delaunay polytopes} associated to the \emph{holes}.  See Section~4 of Conway--Sloane~\cite{CS}.    The \emph{holes} of a lattice are those points that are at maximal distance from lattice points.  A Delaunay polytope associated to a hole is the convex hull of the lattice points closest to the hole.

\section{The $n=1$ case of Theorem~\ref{main}} \label{rank 1}

Theorem~\ref{main} in the case $n=1$ states that   $\mathcal{H}_1(R)$ is the Cayley graph $\mathcal{C}$ of $\Gamma_1(R)$ with respect to the generating set $\set{ \lambda_r := (r,1) \mid r \in R}$.  This generating set is, in fact,  profligate---$\set{\lambda_0, \lambda_1}$ suffices to generate $\Gamma_1(R)$.  This case includes $\Gamma_1 = \Z \wr \Z$ and lamplighters $\Gamma_1(m) = (\Z/m\Z) \wr \Z$.  

\begin{proof}[Proof of Theorem~\ref{main} for $n=1$ (cf.\ \cite{BaW, BW, Woess})]
An  element  of $\Gamma_1(R) = R[x,x^{-1}]\rtimes \Z$ is a pair $(f, k)$  where $k \in \Z$ and $f= \sum f_jx^j$  with each $f_j \in R$ and only finitely many  are non-zero.  Recall from Lemma~\ref{addresses} that vertices in $\mathcal{T}_R$ are uniquely specified by their \emph{addresses}---pairs  consisting of  a finitely supported sequence of elements of $R$ (the edge-labels on the path proceeding downwards from the vertex) and an integer (the height).      

Let $\Phi$ be the bijection between $\Gamma_1(R)$ and the vertices of $\mathcal{H}_1(R)$ that sends $(f, k)$ to the pair of vertices $(u,v)$ with addresses   $((f_{k}, f_{k+1}, f_{k+2}, \ldots), -k)$  and $((f_{k-1}, f_{k-2}, f_{k-3}, \ldots),k)$, respectively.    So, in effect, $\Phi$ splits the bi-infinite sequence  of coefficients of $f$ apart at $k$ to give  two infinite sequences as shown in the middle of Figure~\ref{action on Gamma1}.  The sequence at the locations shaded pink give the address of $u$ and that shaded green gives the address of $v$.  

In $\mathcal{C}$, the edge labeled $\lambda_r$ emanating from $(f, k)$ leads to $(f, k)\lambda_r=(f+rx^{k}, k+1)$, which is mapped by $\Phi$ to  $(u',v')$ where $u'$ and $v'$ have addresses  $((  f_{k+1}, f_{k+2}, \ldots), -k-1)$ and $((f_{k} +r, f_{k-1}, f_{k-2},   \ldots), k+1)$, respectively---see the top of Figure~\ref{action on Gamma1}.  So, as $r$ varies over $R$,  $(u',v')$ varies  over all the vertices adjacent to $(u,v)$ that are reached by moving along the (unique) downwards edge in $\mathcal{T}_R$ emanating from $u$ and moving   along one of the $R$-indexed edges that emanate upwards from $v$.  

The inverse of  ${\lambda_r} = (r,1)$ is $(-rx^{-1}, -1)$ since $$(r,1)(-rx^{-1}, -1) \ = \ (r+(-rx^{-1})x^{1}, 1-1) \ = \ (0,0).$$

So, similarly, the family $(f, k){\lambda_r}^{-1} = (f-rx^{k-1}, k-1)$ with $r$ ranging over $R$,  is mapped by $\Phi$ to  $(u'',v'')$  where $u''$ and $v''$ have addresses  $((  f_{k-1}-r, f_{k}, f_{k+1}, \ldots), -k+1)$ and $((f_{k-2}, f_{k-3},   \ldots), k-1)$, respectively---see the bottom of Figure~\ref{action on Gamma1}.  These are the vertices obtained by moving along the one downwards edge in $\mathcal{T}_R$ from $v$ and moving from $u$  upwards  along one of the $R$-indexed family of edges.   

So, vertices  that are joined by an edge in $\mathcal{C}$ are mapped by   $\Phi$ to vertices  that are joined by an edge in $\mathcal{H}_1(R)$.  Moreover, every pair of vertices  that are joined by an edge in $\mathcal{H}_1(R)$ can be reached in this way.     So  $\Phi$ extends to a graph-isomorphism $\mathcal{C} \to \mathcal{H}_1(R)$, completing our proof.  
\end{proof}

 \begin{figure}[ht]
    \psfrag{e}{\tiny{${f_{k-5}}$}} 
   \psfrag{d}{\tiny{${f_{k-4}}$}} 
   \psfrag{c}{\tiny{${f_{k-3}}$}} 
   \psfrag{b}{\tiny{${f_{k-2}}$}} 
  \psfrag{a}{\tiny{${f_{k-1}}$}} 
  \psfrag{0}{\tiny{${f_{k}}$}} 
    \psfrag{+}{\tiny{${f_{k}}+r$}} 
  \psfrag{-}{\tiny{${f_{k-1}}-r$}} 
 \psfrag{1}{\tiny{${f_{k+1}}$}} 
  \psfrag{2}{\tiny{${f_{k+2}}$}} 
  \psfrag{3}{\tiny{${f_{k+3}}$}} 
  \psfrag{4}{\tiny{${f_{k+4}}$}} 
  \psfrag{5}{\tiny{${f_{k+5}}$}} 
   \psfrag{t}{${t}$}
      \psfrag{g}{$g$:}
         \psfrag{l}{$\lambda_r$}
   \psfrag{w}{$\lambda^{-1}_r$}
 \psfrag{p}{$g\lambda_r$:}
   \psfrag{m}{$g\lambda^{-1}_r$:}
 \centerline{\epsfig{file=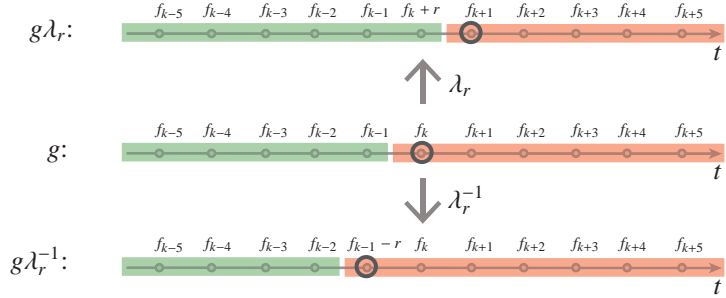}} \caption{Here we use the lamplighter description of $\Gamma_1$ to illustrate right-multiplication by the generators $\lambda_r$ and their inverses.  The middle line represents $g = (f,k)$ and the top and bottom represent $g \lambda_r$ and $g \lambda^{-1}_r$, respectively.} \label{action on Gamma1}
\end{figure}

\begin{remark}
Perhaps the one subtlety in the above proof is that the edge in $\mathcal{T}_R$ from $v$ to $v'$ is labeled by  $f_{k} +r$.  The first guess one might make  is that it would be the edge labeled $r$.   But that  would not work because $(u',v')$ has to have some ``memory'' of $f_k$, else there would be no way for $\Phi^{-1}((u',v') \lambda_r^{-1})$ to equal $\Phi^{-1}(u,v)$. 
\end{remark}

\begin{remark}
In this rank-$1$ case we could use any group $G$ in place of the ring $R$, and identify a  Cayley graph of the (restricted) wreath product $G \wr \Z$ as a horocyclic product.  Specifically, view  elements of $G \wr \Z$ as pairs $(p, k)$ where $k \in \Z$ and $p$ is a finitely supported function $\Z \to G$, and let $p_g$ denote  the map sending $1 \mapsto g$ and $i\mapsto 1_G$ for all $i\neq 1$.  Then the Cayley graph of   $G \wr \Z$
  with respect to the generating set $\set{ \lambda_g := (p_g,1) \mid g \in G}$ is the horocyclic product of two $G$-branching trees.   This appears to break down in higher rank where we would need $G$ to be abelian (e.g.\ to define the lamplighter description in Section~\ref{Gamma_2 lamplighter}).
  \end{remark}

\section{The $n=2$ case of Theorem~\ref{main}} \label{rank 2}

In this section we will prove Theorem~\ref{main} when $n=2$:  the 1-skeleton of $\mathcal{H}_2(R)$ is the Cayley graph  of $\Gamma_2(R)$ with respect to the generating set    $$\set{   \left. \lambda_r := (r, \e_0), \ \mu_r := (r, \e_1), \ \nu_r := \lambda_r {\mu_r}^{-1}  \, \right| \, r \in R \, }.$$

This case includes Baumslag and Remeslennikov's metabelian group, which is $\Gamma_2$.

\subsection{A lamplighter model for $\Gamma_2(R)$.}  \label{Gamma_2 lamplighter}

Recall that $$\Gamma_2(R)  \ = \ R \left[ x, x^{-1}, (1+x)^{-1}\right] \rtimes \Z^{2}$$ where, if the $\Z^2$-factor is $\langle t, s \rangle$, the  actions of $t$ and $s$ are multiplication by $x$ and $1+x$, respectively. 

We will use a \emph{lamplighter description} of $\Gamma_2$  developed  from \cite{BNW} and \cite{CRwithcorrection}.   A lamplighter is located at a  \emph{lattice point} in a skewed rhombic $\Z^2 = \langle t,s \rangle$  grid, as in Figure~\ref{propagation}. (The lattice points are the vertices of the tessellation of the plane by unit equilateral triangles.)  Each vertex has six closest neighbours---one in each of what we will call the $s$-, $s^{-1}$-, $t$-, $t^{-1}$-, $st^{-1}$- and $s^{-1}t$-directions---and can be specified using $t$- and $s$-coordinates.    A \emph{configuration}  $\mathcal{K}$  is a finitely supported assignment of an element of $R$ to each  lattice point.   
 
Figure~\ref{propagation} shows six examples of configurations where $R = \Z$.  Vertices where no element of $R$ is shown should be understood to be assigned zeroes.  As an example of the terminology in action, the integer at $(-2,1)$ in grid (5) is $4$ and its neighbours in the  $s$-, $s^{-1}$-, $t$-, $t^{-1}$-, $st^{-1}$- and $s^{-1}t$-directions are  $0$, $2$, $6$, $1$, $0$, and $-4$, respectively.  

We define  an equivalence relation $\sim$ on configurations by setting $\mathcal{K} \sim \mathcal{K}'$ when there is a finite sequence of configurations starting with $\mathcal{K}$ and ending with $\mathcal{K}'$ in which each configuration  differs from the next only in one  triangle of adjacent ring elements which is $\T{a}{b}{c}$ in one and is $\T{a-r}{b+r}{c+r}$ for some $r \in R$ in the other.  The six integer-configurations shown in Figure~\ref{propagation} are all equivalent, for example.

 \begin{figure}[ht]
 \psfrag{1}{\tiny{${1}$}}   \psfrag{2}{\tiny${2}$}  \psfrag{3}{\tiny${3}$} \psfrag{4}{\tiny${4}$}   \psfrag{5}{\tiny${5}$}   \psfrag{6}{\tiny${6}$}
  \psfrag{7}{\tiny${7}$}  \psfrag{u}{\tiny${10}$} 
  \psfrag{x}{\tiny${2-2}$}
 \psfrag{a}{\tiny${-1}$}   \psfrag{b}{\tiny${-2}$}  \psfrag{c}{\tiny${-3}$} \psfrag{d}{\tiny${-4}$}   \psfrag{e}{\tiny${-5}$}   \psfrag{f}{\tiny${-6}$} 
  \psfrag{y}{\tiny{$4+1$}}
  \psfrag{t}{${t}$}
  \psfrag{s}{${s}$}
    \psfrag{A}{(1)}     \psfrag{B}{(2)}     \psfrag{C}{(3)}    \psfrag{D}{(4)}    \psfrag{E}{(5)}    \psfrag{F}{(6)}
 \centerline{\epsfig{file=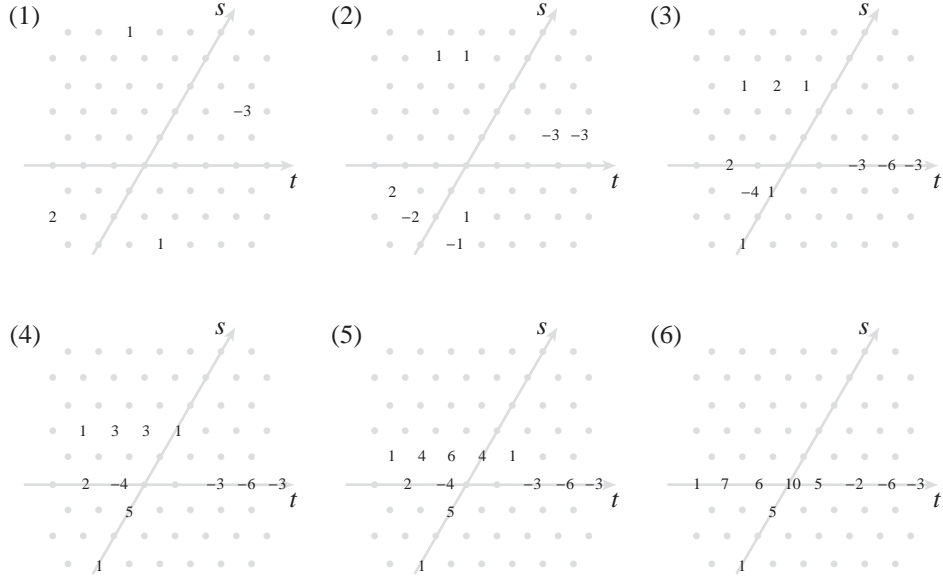}} \caption{An example of propagation to a configuration supported on $L_{0,0}$.} \label{propagation}
\end{figure}

An element $f = \sum_{i,j \in \Z} n_{i, j} x^{i}(1+x)^{j}$ of $R \left[ x, x^{-1}, (1+x)^{-1}\right]$ corresponds to 
the  configuration which  has  $n_{i, j}$ at $(i, j)$ for all $i,j \in \Z$.  A  motivating result for these definitions is---

\begin{lemma}  \label{polynomials versus configurations}
 Two such polynomials represent the same element of  $R \left[ x, x^{-1}, (1+x)^{-1}\right]$ if and only if their corresponding configurations are equivalent.  
 \end{lemma}

\begin{proof}
The relations in $R \left[ x, x^{-1}, (1+x)^{-1}\right]$ are generated by  $(1+x)$ being the sum of the terms $1$ and $x$ in a manner that corresponds to the relations between configurations being generated by altering triangles of entries.  Indeed, multiplying $(1+x) =1+x$ through by $rx^{i}(1+x)^j$ gives $rx^{i}(1+x)^{j+1}=  rx^{i}(1+x)^{j} + rx^{i+1}(1+x)^{j}$, which corresponds to $\T{a+r}{b}{c} \sim \T{a}{b+r}{c+r}$ at a suitably located triangle of entries in a configuration.
\end{proof} 

The element  $g   =    (f, (k,l))   \in   \Gamma_2(R)$  corresponds to the lamplighter being located at $(k,l)$ and the configuration being that associated to $f$.   
 
 An appealing feature of this model is how it elucidates the way in which  $\Gamma_1(R)$ sits inside $\Gamma_2(R)$ (e.g. $\Z \wr \Z$ sits inside Baumslag and Remeslennikov's  group $\Gamma_2$) as the elements for which the lamplighter is on the $t$-axis and the configuration is equivalent to one that is supported on the $t$-axis.   

\begin{defn} \label{half-plane}
Using $t$- and $s$-coordinates, define the half-planes 
\begin{align*}
H_m^{\infty}    \  := \  & \set{ (p,q) \mid p+q \geq m },  \\
H_m^{0}    \ :=  \  & \set{ (p,q) \mid p \leq m },  \\ 
H_m^{1}    \ := \  & \set{ (p,q) \mid q \leq m }.    
\end{align*} 

 For example, Figure~\ref{Projections} displays  $H^{\infty}_{h_0+h_1}$, $H^{0}_{h_0-1}$ and $H^{1}_{h_1-1}$.  
 
 Our analyses will involve finding opportune representatives  in the equivalence classes of given configurations.  Indeed, we will in some instances (in Section~\ref{n=2 proof section}) be concerned  only with the part of a configuration in some  half-plane.  The following definition will then be useful.  
 
 \emph{Propagating to level $\ell$ in $H^{\infty}_m$} means converting a configuration to an equivalent configuration such that  the only non-zero entries in  $H^{\infty}_m$ are on the line with $s$-coordinate $\ell$.  This can always be done by moving the entries in  $H^{\infty}_m$ that are  above that line by using $\T{a}{b}{c}  \sim \T{0}{a+b}{a+c}$  and moving those below by using  $ \T{a}{b}{c} \sim \T{a+c}{b-c}{0}$.  
\emph{Propagating to level $\ell$ in $H^{0}_m$}  means converting to an equivalent configuration such that  the  only non-zero entries  in $H^{0}_m$  are on the line with $s$-coordinate $\ell$.  This can be done  using $\T{a}{b}{c}  \sim \T{0}{a+b}{a+c}$  and    $ \T{a}{b}{c} \sim \T{a+b}{0}{c-b}$  for entries above and below the line, respectively. 
And \emph{propagating to level $\ell$ in $H^{1}_m$}  means converting to an equivalent configuration such that  the  only non-zero entries  in
 $H^{1}_m$  are on the line with $t$-coordinate $\ell$.  This can be done using $\T{a}{b}{c} \sim \T{a+b}{0}{c-b}$ and  $\T{a}{b}{c}  \sim \T{a+c}{b-c}{0}$  for entries on the left and the right of the line, respectively.  
 
 In each case, propagation produces  a finitely supported sequence, namely the entries in level $\ell$ of the half-plane concerned.
   For example,  in Figure~\ref{propagation}  propagating the integer-configuration  (1) to level $0$  in $H^{\infty}_{0}$, $H^{0}_{-1}$ and $H^{1}_{-1}$  yields configurations which can be read off (6), specifically,   $10, 5, -2, -6, -3, 0, 0, \ldots$ in  $H^{\infty}_{0}$,  $6, 7, 1, 0, 0, \ldots$  in $H^{0}_{-1}$, and $5, 0, 1, 0, 0, \ldots$ in $H^{1}_{-1}$.   And in Figure~\ref{example Phi}, the configuration in the centre grid propagated to level $0$ yields  $5, 3, 4, 2, 0, 0, \ldots$ in $H^{\infty}_{3}$, $18, 5, 1, 0, 0, \ldots$ in $H^{0}_{0}$, and $2, 3, 0, 1, 0, 0, \ldots$ in $H^{1}_{1}$.  \end{defn}

\begin{figure}[ht]
\psfrag{s}{$s$} \psfrag{t}{$t$} \psfrag{i}{$h_0$} \psfrag{j}{$h_1$}   \psfrag{+}{$h_0+h_1$} 
\psfrag{8}{$\mathbf{a}^{\infty}$}
 \psfrag{0}{$\mathbf{a}^{0}$}
  \psfrag{1}{$\mathbf{a}^{1}$}
  \psfrag{a}{$\mathbf{b}^{\infty}$}
 \psfrag{b}{$\mathbf{b}^{0}$}
  \psfrag{c}{$\mathbf{b}^{1}$}
    \psfrag{x}{$H^{\infty}_{h_0+h_1}$}
 \psfrag{y}{$H^{0}_{h_0-1}$}
  \psfrag{z}{$H^{1}_{h_1-1}$}
 \centerline{\epsfig{width= 5 in, file=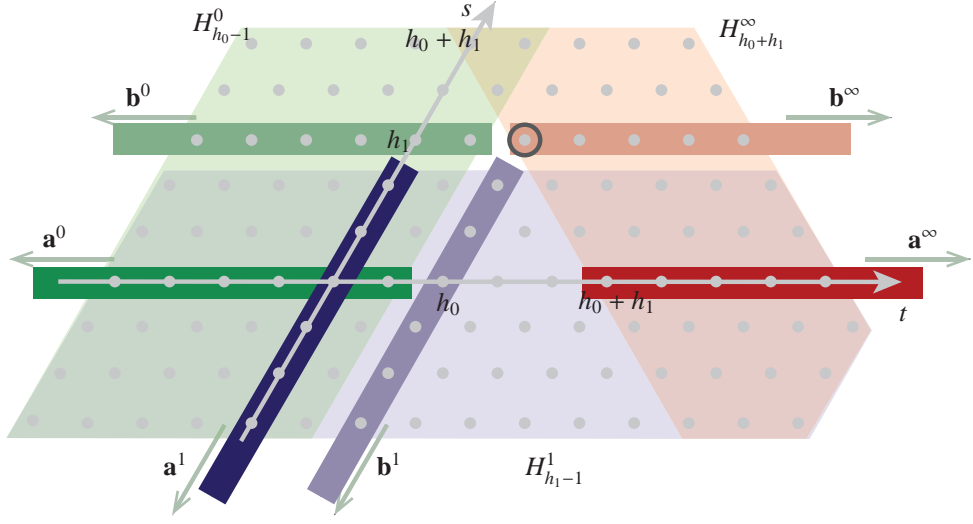}} \caption{Propagation in  the half-planes $H^{\infty}_{h_0+h_1}$, $H^{0}_{h_0-1}$ and   $H^{1}_{h_1-1}$.  Propagation to levels $h_1$, $h_1$ and $h_0$, respectively, is illustrated  using lighter colors. Propagation to level $0$ in each half-plane  is illustrated using darker colors. } \label{Projections}
\end{figure}

The following properties of propagation  may at first seem surprising because it is not immediately apparent that   
the entries outside $H^{\ast}_m$  are of no consequence for the sequence produced by propagation. 

\begin{lemma}  \label{projection properties}
For $\ast = \infty, 0,1$ and for all $\ell, \ell' \in \Z$ the following hold.  
\vspace*{-2mm}
\begin{enumerate}
\item  
Any two equivalent configurations which are both zero everywhere in $H^{\ast}_m$ aside from level $\ell$, are in fact  equal on level $\ell$ in $H^{\ast}_m$.  (So propagation of a configuration to level $\ell$ in $H^{\ast}_m$ determines a unique sequence and propagating any  two  equivalent configurations to level $\ell$ in $H^{\ast}_m$ produces the same sequence.)  
\item If propagating a configuration $\mathcal{K}$ to  level $\ell$ in $H^{\ast}_m$ produces the sequence $a_1, a_2, \ldots$, then  $a_p$, for $p=1, 2, \ldots$, depends only on the restriction of $\mathcal{K}$ to 
$$\begin{cases}
H^{\infty}_{m+p-1} & \mbox{ if } \ast = \infty  \\
H^{0}_{m-p+1} & \mbox{ if } \ast = 0 \\
H^{1}_{m-p+1} & \mbox{ if } \ast = 1.
\end{cases}$$  
\item The following defines a bijection on the set of  finitely supported integer sequences.  Given such a sequence, take the configuration which is everywhere-zero aside from level $\ell$ of $H^{\ast}_m$  where one reads the sequence, and obtain a new sequence by  propagating to level $\ell'$ in $H^{\ast}_m$.  Indeed, this map is inverted by propagating back to  level $\ell$.     
\end{enumerate}
\end{lemma}

\begin{proof}
We will explain only the case $\ast = \infty$. The cases $\ast = 0, 1$ are similar.  

For (\emph{i}), recall that the equivalence relation on configurations is generated by equivalences in which a triangle of only three adjacent entries is altered.   Such alterations   do not change the  sequence obtained by propagating to level $\ell$ in  $H^{\infty}_m$  by moving those above the level  using $\T{a}{b}{c}  \sim \T{0}{a+b}{a+c}$     and moving those below by using   $ \T{a}{b}{c} \sim \T{a+c}{b-c}{0}$.  
Consideration of the directions in which entries are moved  by these two types of equivalences  leads  to (\emph{ii}).  For (\emph{iii}) observe that the result is true when $\abs{\ell-\ell'}=1$.  
\end{proof}

\begin{cor}\label{basis}
For all $k,l  \in \Z$, each configuration is equivalent to a unique configuration   supported on
 $$L_{k,l} \ :=  \ \set{ \,  (i,l) \,  \mid \,  i \in \Z \, } \ \cup \ \set{ \, (k,l-1), (k,l-2), \ldots   \, },$$
 specifically, that obtained by simultaneously propagating to level  $l$ in $H^{\infty}_{k+l}$ and $H^{0}_{k-1}$ and to level  $k$ in   $H^{1}_{l-1}$.
\end{cor}

 \begin{figure}[ht]
 \psfrag{1}{\tiny{${1}$}}   \psfrag{2}{\tiny${2}$}  \psfrag{3}{\tiny${3}$} \psfrag{4}{\tiny${4}$}   \psfrag{5}{\tiny${5}$}   \psfrag{6}{\tiny${6}$}
  \psfrag{7}{\tiny${7}$}  \psfrag{u}{\tiny${10}$} \psfrag{l}{\tiny${11}$} \psfrag{?}{\tiny${18}$}  \psfrag{m}{\tiny${25}$}
  \psfrag{x}{\tiny${2-2}$}
 \psfrag{a}{\tiny${-1}$}   \psfrag{b}{\tiny${-2}$}  \psfrag{c}{\tiny${-3}$} \psfrag{d}{\tiny${-4}$}   \psfrag{e}{\tiny${-5}$}   \psfrag{f}{\tiny${-6}$} 
  \psfrag{y}{\tiny{$4+1$}}
  \psfrag{t}{${t}$}
  \psfrag{s}{${s}$}
    \psfrag{A}{(1)}     \psfrag{B}{(2)}     \psfrag{C}{(3)}    \psfrag{D}{(4)}    \psfrag{E}{(5)}    \psfrag{F}{(6)}
    \psfrag{i}{$i$} \psfrag{j}{$j$}   \psfrag{+}{$i+j$} 
\psfrag{X}{$\mathbf{a}^{\infty}$}
 \psfrag{Y}{$\mathbf{a}^{0}$}
  \psfrag{Z}{$\mathbf{a}^{1}$}
  \psfrag{p}{$\mathbf{b}^{\infty}$}
 \psfrag{q}{$\mathbf{b}^{0}$}
  \psfrag{r}{$\mathbf{b}^{1}$}
    \psfrag{x}{$H^{\infty}_{3}$}
 \psfrag{y}{$H^{0}_{0}$}
  \psfrag{z}{$H^{1}_{1}$} 
     \centerline{\epsfig{file=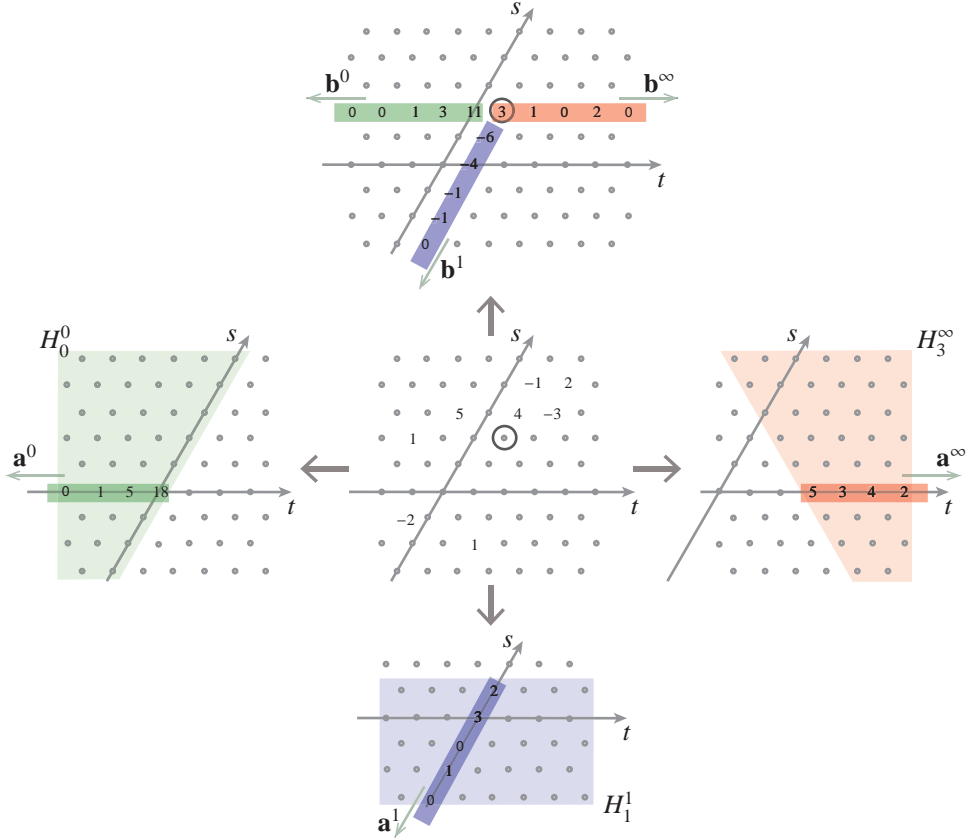}} \caption{An example of a calculation of $\Phi(g)$, where $g$ is the element of $\Gamma_2$ represented on the central grid.  The lamplighter is at   $(1,2)$, so $h_{\infty} = -1-2=-3$, $h_{0}=1$, and $h_1=2$.  The right, left, and lower grid illustrate the calculation of $\mathbf{a}^{\infty} = (5,3,4,2,0,0, \ldots)$,  $\mathbf{a}^{0} = (18,5,1,0,0, \ldots)$, and  $\mathbf{a}^{1} = (2,3,0,1,0,0, \ldots)$, respectively, by propagation to level $0$ in $H_3^{\infty}$, $H_0^{0}$, and $H_1^{1}$.   The upper grid illustrates a configuration   which is  supported on $L_{1,2}$, is equivalent to that of the central grid, and  yields the sequences   $\mathbf{b}^{\infty} = (3,1,0,2,0,0, \ldots)$,  $\mathbf{b}^{0} = (11,3,1,0,0, \ldots)$, and  $\mathbf{b}^{1} = (-6,-4,-1,-1,0,0, \ldots)$, which feature in our proof  of case $n=2$ of Theorem~\ref{main}. } \label{example Phi}
\end{figure}
 
In the light of Lemma~\ref{polynomials versus configurations}, when  $k= l=0$ this says that
 $$\set{ \left. \,  1, x^{j},  x^{-j}, (1+x)^{-j} \, \right| \,  j   = 1, 2, \ldots   \, }$$ is a basis for $R[x,x^{-1}, (1+x)^{-1}]$ over $R$.  (This is a special case of Lemma~\ref{basis lemma}.)
 
Figure~\ref{propagation} shows an example of such a propagation with $k=\ell=0$, and  the transition from the central grid to the top grid in Figure~\ref{example Phi}  is an example with $k=1$ and $\ell =2$.

\subsection{Proof of Theorem~\ref{main} in the case $n=2$} \label{n=2 proof section}
 
We are now ready to show that the 1-skeleton of   $\mathcal{H}_2(R)$  is the Cayley graph $\mathcal{C}$ of   $\Gamma_2(R)$ with respect to   $$\set{   \left. \lambda_r := (r, \e_0), \ \mu_r := (r, \e_1), \ \nu_r := \lambda_r {\mu_r}^{-1}  \, \right| \, r \in R \, }.$$ 

We will denote a vertex in $\mathcal{H}_2(R)$  by a triple of vertices in $\mathcal{T}_R$, each designated by their addresses in the sense of Lemma~\ref{addresses}.
First we will establish a bijection $\Phi$ from $\Gamma_2(R)$ to the vertices of $\mathcal{H}_2(R)$, defined by sending $g = (f, (h_0,h_1)) \in \Gamma_2(R)$ to the vertex $((\mathbf{a}^{\infty},h_{\infty}), (\mathbf{a}^0, h_0), (\mathbf{a}^1,h_1) )$ found as follows.  Represent   $f$ using the  lamplighter model as some configuration $\mathcal{K}$.  Let $h_{\infty} =  -h_0-h_1 $. Let $\mathbf{a}^{\infty}$, $\mathbf{a}^{0}$, and $\mathbf{a}^{1}$ be the sequences obtained by (independently) propagating  $\mathcal{K}$ to  level $0$ in the half-planes $H^{\infty}_{h_0+h_1}$,  $H^{0}_{h_0-1}$, and $H^{1}_{h_1-1}$,  respectively---see Figure~\ref{Projections} for a general illustration and  Figure~\ref{example Phi} for a particular example.  

Here is why $\Phi$ is a bijection.  Let  $\mathcal{K}'$  be the configuration of Corollary~\ref{basis} that is equivalent to $\mathcal{K}$  and is supported  on $L_{h_0,h_1}$.  As that corollary points out,  $\mathcal{K}'$ is determined by the sequences  $\mathbf{b}^{\infty}$, $\mathbf{b}^{0}$, and $\mathbf{b}^{1}$ obtained from $\mathcal{K}$ by  propagating  $H^{\infty}_{h_0+ h_{1}}$ and $H^{0}_{h_0-1}$ to level $h_1$,  and    $H^{1}_{h_1-1}$ to level  $h_0$.  But, given $h_0$ and $h_1$, the bijection of Lemma~\ref{projection properties}(\emph{iii})  tells us that $\mathbf{b}^{\infty}$, $\mathbf{b}^{0}$, and $\mathbf{b}^{1}$ are determined by (and determine) $\mathbf{a}^{\infty}$, $\mathbf{a}^{0}$, and $\mathbf{a}^{1}$, respectively. 
So, given any vertex $\mathbf{v} = ((\mathbf{a}^{\infty},h_{\infty}), (\mathbf{a}^0, h_0), (\mathbf{a}^1,h_1))$ in $\mathcal{H}_2(R)$, there is a unique $g= (f, (h_0,h_1))$ such that $\Phi(g) = \mathbf{v}$: specifically, take the $f$ corresponding to $\mathcal{K}'$.
(This is a special case of Proposition~\ref{it's a bijection}.) 
 
  Next we claim that for all $r \in R$,    
 \begin{align*}
\Phi(g\lambda_r) \ = \ & \left(   \,    \left(  \rule{0pt}{13pt} \left(a_2^{\infty}, a_3^{\infty}, \ldots \right), h_{\infty}-1 \right),  \  \     \left( \rule{0pt}{13pt}  \left(r + \alpha, a_1^0, a_2^0, \ldots \right), h_0+1 \right),    \ \  \left(  \rule{0pt}{13pt} \mathbf{a}^1, h_1\right) \,     \rule{0pt}{16pt}\right), \\
\Phi(g{\lambda_r}^{-1}) \ = \ & \left(   \,    \left(  \rule{0pt}{13pt} \left(-r+ \alpha', a_1^{\infty}, a_2^{\infty}, \ldots \right), h_{\infty}+1 \right),  \  \     \left( \rule{0pt}{13pt}  \left(a_2^0, a_3^0, \ldots \right), h_0-1 \right),    \ \  \left(  \rule{0pt}{13pt} \mathbf{a}^1, h_1\right) \,     \rule{0pt}{16pt}\right), \\
\Phi(g\mu_r) \ = \ & \left(   \,   \left(  \rule{0pt}{13pt} \left(a_2^{\infty}, a_3^{\infty}, \ldots \right), h_{\infty}-1 \right), \ \   \left( \rule{0pt}{13pt}  \mathbf{a}^0, h_0 \right),    \ \  \left(  \rule{0pt}{13pt}  \left( (-1)^{h_0}r + \beta, a^1_1, a^1_2, \ldots \right), h_1+1\right)  \,     \rule{0pt}{16pt}\right),   \\
\Phi(g{\mu_r}^{-1}) \ = \ & \left(   \,   \left(  \rule{0pt}{13pt} \left(-r+\beta', a_1^{\infty}, a_2^{\infty}, \ldots \right), h_{\infty}+1 \right), \ \   \left( \rule{0pt}{13pt}  \mathbf{a}^0, h_0 \right),    \ \  \left(  \rule{0pt}{13pt}  \left( a^1_2, a^1_3, \ldots \right), h_1-1\right)  \,     \rule{0pt}{16pt}\right),  \\
\Phi(g\nu_r) \ = \ & \left(   \,  \left(  \rule{0pt}{13pt} \mathbf{a}^{\infty}, h_{\infty} \right), \ \   \left( \rule{0pt}{13pt}  \left( r+\gamma, a_1^0, a_2^0, \ldots \right), h_0+1 \right),    \ \  \left(  \rule{0pt}{13pt} \left(a_2^1, a_3^1, \ldots \right), h_1-1\right) \,     \rule{0pt}{16pt}\right),  \\
\Phi(g{\nu_r}^{-1}) \ = \ & \left(   \,  \left(  \rule{0pt}{13pt} \mathbf{a}^{\infty}, h_{\infty} \right), \ \   \left( \rule{0pt}{13pt}  \left(  a_2^0, a_3^0, \ldots \right), h_0-1 \right),    \ \  \left(  \rule{0pt}{13pt} \left((-1)^{h_0}r+ \gamma', a_1^1, a_2^1, \ldots \right), h_1+1\right) \,     \rule{0pt}{16pt}\right) 
\end{align*}
where  $\alpha$, $\alpha'$, $\beta$, $\beta'$,  $\gamma$, and  $\gamma'$    depend only on $g$ (and not on $r$).

 \begin{figure}[ht]
 \psfrag{1}{\tiny{$\lambda_r$}}
  \psfrag{2}{\tiny{$\lambda^{-1}_r$}}
 \psfrag{5}{\tiny{$\mu_r$}}
  \psfrag{6}{\tiny{$\mu^{-1}_r$}}
 \psfrag{4}{\tiny{$\nu_r$}}
  \psfrag{3}{\tiny{$\nu^{-1}_r$}}
\psfrag{X}{\tiny{$\mathbf{b}^{\infty}$}}
 \psfrag{Y}{\tiny{$\mathbf{b}^{0}$}}
  \psfrag{Z}{\tiny{$\mathbf{b}^{1}$}}
  \psfrag{P}{\tiny{$\mathbf{a}^{\infty}$}}
 \psfrag{Q}{\tiny{$\mathbf{a}^{0}$}}
  \psfrag{R}{\tiny{$\mathbf{a}^{1}$}}
  \psfrag{x}{\tiny{$\bar{\mathbf{b}}^{\infty}$}}
 \psfrag{y}{\tiny{$\bar{\mathbf{b}}^{0}$}}
  \psfrag{z}{\tiny{$\bar{\mathbf{b}}^{1}$}}
  \psfrag{p}{\tiny{$\bar{\mathbf{a}}^{\infty}$}}
 \psfrag{q}{\tiny{$\bar{\mathbf{a}}^{0}$}}
  \psfrag{r}{\tiny{$\bar{\mathbf{a}}^{1}$}}
   \psfrag{s}{\tiny{$s$}}
  \psfrag{t}{\tiny{$t$}}
    \psfrag{i}{\tiny{$h_0$}}
  \psfrag{j}{\tiny{$h_1$}}
    \psfrag{+}{\tiny{$+r$}}
    \psfrag{-}{\tiny{$-r$}}
     \psfrag{A}{\tiny{$(-1)^{h_0}r+\gamma'$}}
     \psfrag{B}{\tiny{$(-1)^{h_0}r+\beta$}}
     \psfrag{C}{\tiny{$r+\alpha$}}
     \psfrag{D}{\tiny{$r+\gamma$}}
     \psfrag{E}{\tiny{$-r+\beta'$}}
     \psfrag{F}{\tiny{$-r+\alpha'$}}
     \centerline{\epsfig{file=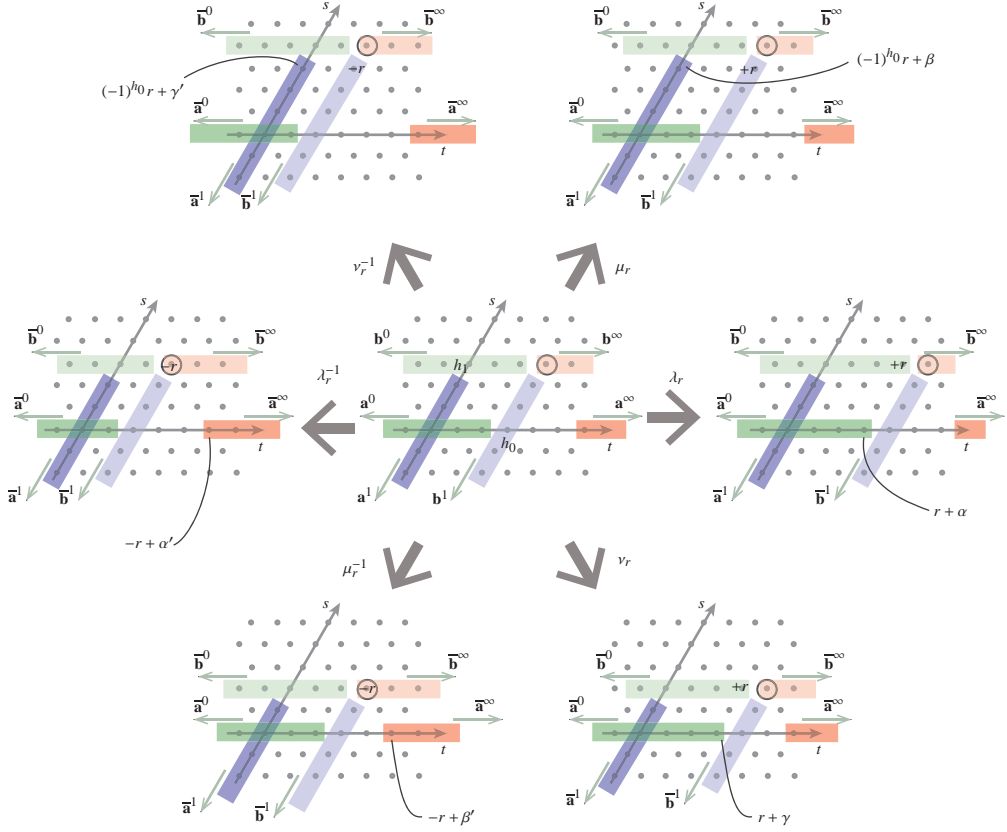}} \caption{Obtaining $\Phi(g\lambda_r)$, $\Phi(g\lambda^{-1}_r)$, $\Phi(g\mu_r)$, $\Phi(g\mu^{-1}_r)$, $\Phi(g\nu_r)$, and $\Phi(g\nu^{-1}_r)$  from $\Phi(g) =((\mathbf{a}^{\infty},h_{\infty}), (\mathbf{a}^0, h_0), (\mathbf{a}^1,h_1))$.  The sequences associated to the former are denoted here  by $\bar{\mathbf{a}}^{\infty}$, $\bar{\mathbf{a}}^{0}$, $\bar{\mathbf{a}}^{1}$, $\bar{\mathbf{b}}^{\infty}$, $\bar{\mathbf{b}}^{0}$, and $\bar{\mathbf{b}}^{1}$.  The central grid represents $g$ and the six outer grids represent  $g\lambda_r$,  $g\lambda^{-1}_r$,  $g\mu_r$,  $g\mu^{-1}_r$, $g\nu_r$, and  $g\nu^{-1}_r$, as indicated.} \label{multiply figure}
\end{figure}

As we will see, much of the explanation for these equations is contained in Figure~\ref{multiply figure}.  The central grid represents $g$: the lamplighter is at $(h_0, h_1)$ and the  
 sequences  ${\mathbf{a}}^{\infty}$, ${\mathbf{a}}^{0}$, ${\mathbf{a}}^{1}$, ${\mathbf{b}}^{\infty}$, ${\mathbf{b}}^{0}$, and ${\mathbf{b}}^{1}$ associated to $f$ are obtained from the locations indicated (in the manner set out earlier).  On right-multiplying $g$ by $\lambda_r$, $\lambda^{-1}_r$, $\mu_r$, $\mu^{-1}_r$, $\nu_r$, or $\nu^{-1}_r$, the lamplighter moves as shown and $r$ is added  to or subtracted from one entry in the configuration (also as shown).  The locations from which the   sequences  $\bar{\mathbf{a}}^{\infty}$, $\bar{\mathbf{a}}^{0}$, $\bar{\mathbf{a}}^{1}$, $\bar{\mathbf{b}}^{\infty}$, $\bar{\mathbf{b}}^{0}$, and $\bar{\mathbf{b}}^{1}$ associated to the new configurations  are obtained also shift as shown. 
 
Here is the justification for the first coordinates on the righthand sides of the six equations above.   

Here  is why the first coordinate of $\Phi(g\lambda_r)$ is   $\left(  \rule{0pt}{0pt} \left(a_2^{\infty}, a_3^{\infty}, \ldots \right), h_{\infty}-1 \right)$.   Since $$g \lambda_r \ = \ (f+ r \cdot (h_0,h_1), \  (h_0,h_1) + \mathbf{e}_0) \ = \  (f+ r x^{h_0}  (1+x)^{h_1}, (h_0+1,h_1)),$$  
the representation of $g\lambda_r$ in the lamplighter model  is obtained from that of $g$ by adding $r$ to  the entry in $\mathcal{K}$ at $(h_0,h_1)$  and moving the lamplighter to  $(h_0+1,h_1)$.    The second entry is $h_{\infty}-1$ because $(h_{\infty}-1) +  (h_{0}+1) + h_1 =0$, and $\bar{\mathbf{a}}^{\infty}$ is $\left(a_2^{\infty}, a_3^{\infty}, \ldots \right)$  by Lemma~\ref{projection properties}(\emph{ii}), since the sequence obtained by propagating $H^{\infty}_{h_0+h_1+1}$ to level $0$ is the same as that obtained by propagating $H^{\infty}_{h_0+h_1}$ to level $0$ and discarding the first entry.  
 
The first coordinate of  $\Phi(g\mu_r)$ can be identified likewise.  

Similarly,
since 
\begin{align*}
g \lambda_r^{-1} \ & = (f,(h_0,h_1))(-r\cdot (-\mathbf{e}_0),-\mathbf{e}_0)\\ & =\ (f- r \cdot (h_0,h_1)\cdot(- \mathbf{e}_0), \  (h_0,h_1) - \mathbf{e}_0) \\
& = \  (f- r x^{h_0-1}  (1+x)^{h_1}, (h_0-1,h_1)),
\end{align*}  

 the representation of $g{\lambda_r}^{-1}$ is obtained by moving the lamplighter left to  $(h_0-1,h_1)$ and subtracting $r$ from the entry there.  We claim that $\Phi(g{\lambda_r}^{-1})$ has first coordinate $$\left(  \rule{0pt}{0pt} \left(-r+ \alpha', a_1^{\infty}, a_2^{\infty}, \ldots \right), h_{\infty}+1 \right)$$ where $\alpha'$ depends only on $g$.  The second entry is $ h_{\infty}+1$ because $(h_{\infty}+1) + (h_0-1) +h_1 =0$.  
 All but the first entry of the sequence $\bar{\mathbf{a}}^{\infty}$ can again be identified by using Lemma~\ref{projection properties}(\emph{ii}). In propagation in  $H^{\infty}_{h_0+h_1-1}$, entries on the boundary line (that through $(h_0+h_1-1,0)$ and $(0,h_0+h_1-1)$) advance only along that line:  they are unchanged as they propagate and  they do not affect any other entries in the resulting sequence.  So the $r$ subtracted from the entry at  $(h_0-1,h_1)$ moves, undisturbed to $(h_0+h_1-1,0)$.    The $\alpha'$ is the first entry in the sequence  when the portion of $\mathcal{K}$ in $H^{\infty}_{h_0+h_1-1}$ is propagated to level $0$.  So it depends only on $g$.   

The first coordinate of  $\Phi(g{\mu_r}^{-1})$ can be identified likewise.  

Since $\nu_r = \lambda_r {\mu_r}^{-1}$,  
the representation of $g{\nu_r}$ is  obtained by  adding $r$ to the entry in $\mathcal{K}$ at $(h_0,h_1)$, moving  the lamplighter to  $(h_0+1,h_1)$, then moving  the lamplighter to  $(h_0+1,h_1-1)$, and then subtracting $r$ from the entry at $(h_0+1,h_1-1)$.   Equivalently, it is obtained by moving  the lamplighter to  $(h_0+1,h_1-1)$ and  adding $r$ to the entry   at  $(h_0,h_1-1)$.   So  the first coordinate of $\Phi(g{\nu_r})$ is  $\left( \mathbf{a}^{\infty}, h_{\infty} \right)$:  
the second entry is $h_{\infty}$ because $h_{\infty} + (h_0+1) + (h_0-1)=0$ and $\bar{\mathbf{a}}^{\infty}= {\mathbf{a}}^{\infty}$ because $\bar{\mathbf{a}}^{\infty}$ and ${\mathbf{a}}^{\infty}$ are both obtained by propagating in $H^{\infty}_{h_0+h_1}$, and the altered entry in the configuration is outside $H^{\infty}_{h_0+h_1}$.  

The first coordinate of  $\Phi(g{\nu_r}^{-1})$ is $\left( \mathbf{a}^{\infty}, h_{\infty} \right)$ likewise.

The entries in the second and third coordinates  are explained analogously except for $\Phi(g{\mu_r})$ and $\Phi(g{\nu^{-1}_r})$, where there is an added complication.   When, in the case of $\Phi(g{\mu_r})$,  the $r$ added at $(h_0, h_1)$ is propagated to  $(0, h_1 )$ it changes sign with each step and so becomes  $(-1)^{h_0}r$.  Similarly,  for $\Phi(g{\nu_r}^{-1})$, the $r$ subtracted from $(h_0-1, h_1)$  changes sign with each step as it  propagates to $(0, h_1)$, and so also   becomes  $(-1)^{h_0-1}(-r) = (-1)^{h_0}r$.

Finally, we explain why $\Phi$ extends to an isomorphism from the Cayley graph $\mathcal{C}$ to the 1-skeleton of $\mathcal{H}_2(R)$.

Suppose $g \in \Gamma_2(R)$.   The set of vertices $\mathcal{V}$ in $\mathcal{H}_2(R)$  that are reached by traveling from   $\Phi(g)$ along  a single edge partitions into six subsets:  travel along the unique downwards edge  in one  coordinate-tree,  travel upwards  along one of an $R$-indexed family of edges in another, and    remain  stationary in the last.   Since $\alpha$, $\alpha'$, $\beta$, $\beta'$,  $\gamma$, and $\gamma'$ only depend on $g$,   each of $g\lambda_r \mapsto  \Phi(g\lambda_r)$,    $g{\lambda_r}^{-1} \mapsto  \Phi(g{\lambda_r}^{-1})$, $g\mu_r \mapsto  \Phi(g\mu_r)$,    $g{\mu_r}^{-1} \mapsto  \Phi(g{\mu_r}^{-1})$, $g\nu_r \mapsto  \Phi(g\nu_r)$,  and  $g{\nu_r}^{-1} \mapsto  \Phi(g{\nu_r}^{-1})$   is a  map onto one such subset, and together they give a bijection from the neighbours of $g$ in $\mathcal{C}$ to $\mathcal{V}$.
 
There are no double-edges and no edge-loops in either graph:  for the 1-skeleton of $\mathcal{H}_2(R)$ this is straightforward from the definition, and it therefore follows from the above for the Cayley graph.   So $\Phi$ extends to an isomorphism between the two graphs, and this completes our proof.  
 
\begin{remark}
It may be tempting to try to express directly the group multiplication in $\Gamma_2(R)$ in terms of the representations of elements as triples of addresses of vertices in $\mathcal{T}_R$.   It is striking how spectacularly awkward this turns out to be, as the following special case of multiplication by a generator $\zeta \in \set{\lambda_r^{\pm 1}, \mu_r^{\pm 1}, \nu_r^{\pm 1}}$ illustrates.

We have $\Phi(g) = ((\mathbf{a}^{\infty},h_{\infty}), (\mathbf{a}^0, h_0), (\mathbf{a}^1,h_1))$.  To find 
$\Phi(g \zeta)$ we call on the sequences $\mathbf{b}^{\infty}$, $\mathbf{b}^{0}$ and $\mathbf{b}^{1}$.
Since the propagation (of the bijection established in Lemma~\ref{projection properties}(\emph{iii})) in a half-plane proceeds in the manner of Pascal's triangle, we can explicitly
express  $\mathbf{a}^{\ast}$  in terms of   $\mathbf{b}^{\ast}$ and $\mathbf{b}^{\ast}$  in terms of   $\mathbf{a}^{\ast}$:
\begin{align*}
a^{\ast}_p \ = \  &
\begin{cases}
\displaystyle{ \ \sum_{i=0}^m} (-1)^{\delta} b^{\ast}_{p+i} \SB{m}{i}  & \textup{if} \ m \geq 0 \\
\displaystyle{ \ \sum_{i=0}^{\infty}} (-1)^{\epsilon} b^{\ast}_{p+i} \SB{i-1-m}{i}  & \textup{if} \ m < 0,  
 \end{cases}  \\
b^{\ast}_p \ = \  &
\begin{cases}
\displaystyle{ \ \sum_{i=0}^{-m}} (-1)^{\delta} a^{\ast}_{p+i} \SB{-m}{i}  &  \textup{if} \ m \leq 0 \\
\displaystyle{ \ \sum_{i=0}^{\infty}} (-1)^{\epsilon} a^{\ast}_{p+i} \SB{i-1+m}{i}  & \textup{if} \ m > 0 ,
 \end{cases}  
\end{align*}

when $\ast=\infty,0$ and $m=h_1$, $\epsilon=i$, and $\delta=0$, and when $\ast=1$, and $m=h_0$, $\epsilon= |h_0|$, and $\delta=i+|h_0|$. The infinite sums make sense since all but finitely many entries of the sequences $\mathbf{a}^{\ast}$ and $\mathbf{b}^{\ast}$  are zero.

These formulae could be used to express  $\alpha$, $\alpha'$, $\beta$, $\beta'$,  $\gamma$, and  $\gamma'$ in terms of $\mathbf{a}^{\infty}$, $\mathbf{a}^{0}$, $\mathbf{a}^{1}$, $h_0$ and $h_1$:   obtain $\mathbf{b}^{\infty}$, $\mathbf{b}^{0}$, and $\mathbf{b}^{1}$ using the second formula, then shift them and add or subtract $r$ appropriately to get the $\bar{\mathbf{b}}^{\infty}$, $\bar{\mathbf{b}}^{0}$ and $\bar{\mathbf{b}}^{1}$ associated to $\Phi(g \zeta)$, and finally obtain  $\alpha$, $\alpha'$, $\beta$, $\beta'$,  $\gamma$, and  $\gamma'$ using the first formula.

For example, to calculate $\alpha'$ 
first obtain $\mathbf{b}^{\infty}$ and $b_1^{0}$ from  $\mathbf{a}^{\infty}$ and $\mathbf{a}^{0}$  using the second formula with $m=h_1$, 
then let $\bar{\mathbf{b}}^{\infty}=(b^0_1 -r, b^{\infty}_1, b^{\infty}_2, b^{\infty}_3, \ldots)$, 
then obtain $\bar{\mathbf{a}}^{\infty}$ from $\bar{\mathbf{b}}^{\infty}$ using the first formula with $m=h_1$, and then, as $-r + \alpha'=\bar{a}^{\infty}_1$, we have found $\alpha'$.

  The  complexity of the formulae that would result stands in marked contrast to the ``$f_k+r$'' in our proof in Section~\ref{rank 1} of Theorem~\ref{main} in the case where $n=1$.
\end{remark}

\begin{remark} \label{rank-2 polynomial remark}
Given that equivalence classes of configurations correspond to elements of $R[x,x^{-1}, (1+x)^{-1}]$, the above analysis can all be rephrased in terms of polynomials---the point-of-view we will take in the next section. 
In the light of Lemma~\ref{polynomials versus configurations}, Corollary~\ref{basis} amounts to the statement that for each pair $(k,l)\in \z^2$,
 $$\set{ \, \left. x^{k+i}(1+x)^l  \,   \right| \,  i \in \Z \, } \, \cup \, \set{ \, \left. x^k(1+x)^{j+l}   \, \right| \,  j=-1,-2, \cdots \, }$$ is a basis for  $R \left[ x, x^{-1}, (1+x)^{-1}\right]$ over $R$. 

 The sequence  $\mathbf{a}^{\infty}$ lists the coefficients of $x^{0}, x^{1}, \ldots$ in $x^{h_{\infty}}f$, 
when expressed as a linear combination of the basis   $$\set{ \, \left. x^{i}  \,   \right| \,  i \in \Z \, } \, \cup \, \set{ \, \left.  (1+x)^{j}   \, \right| \,  j=-1,-2, \cdots \, }.$$ Likewise,    $\mathbf{a}^{0}$   lists the coefficients of   $x^{-1}, x^{-2}, \ldots$ in $x^{-{h_0}}f$, and $\mathbf{a}^{1}$  lists those of $(1+x)^{-1}, (1+x)^{-2}, \ldots$ in $(1+x)^{-h_1}f$.

If we multiply $f$ by $x^{-h_0} (1+x)^{-h_1}$ to give   $\hat{f}$  (in effect,  shifting the origin from $(0,0)$ to $(h_0,h_1)$), then  $\mathbf{b}^{\infty}$ lists the coefficients of $x^{0}, x^{1}, \ldots$ in $\hat{f}$, and $\mathbf{b}^{0}$ lists the coefficients of $x^{-1}, x^{-2}, \ldots$,  and    $\mathbf{b}^{1}$ lists those of $(1+x)^{-1}, (1+x)^{-2}, \ldots$.
     \end{remark}

\section{The general case of Theorem~\ref{main}} \label{rank n}

The standing assumptions in this section are that $n$  is any fixed positive integer and $R$ is any commutative ring with unity in which $2$, $3$, \ldots, $n-1$ are invertible.     We will prove Theorem~\ref{main} in full generality: the 1-skeleton of $\mathcal{H}_n(R)$ is the specified Cayley graph.
 
 \subsection{Preliminaries}
 
Recall that  
 $$A_n(R) \ = \ R \left[ x, x^{-1}, (1+x)^{-1}, \ldots, (n -1+x)^{-1} \right].$$
  
 The following lemma generalizes Corollary~\ref{basis} and is vital to the proof of Theorem~\ref{main}.   Baumslag \& Stammbach  \cite{BaS} prove a very similar result as do Bartholdi, Neuhauser \& Woess \cite[Section~3]{BNW}. 
 We include a proof for completeness and because this and the    lemmas that follow are where the hypothesis that   $2, 3, \ldots, n-1$ are invertible is used.
 
 \begin{lemma}[\emph{adapted from} Baumslag \& Stammbach, Lemma~2.1, \cite{BaS}]  \label{basis lemma}
 $$\set{ \left. \,  1, x^{j},  x^{-j}, (1+x)^{-j}, \ldots, (n-1+x)^{-j} \, \right| \,  j   = 1, 2, \ldots   \, }$$ is a basis for $A_n(R)$ over $R$.  
\end{lemma} 

\begin{proof}
First we show that the given set spans.  

Suppose $S \subseteq \set{0, 1, \ldots, n-1}$. 

For $l \in S$, let $$\lambda_l  \  :=  \ \prod_{i \in S \ssm \set{l}} (i-l)^{-1},$$ understanding this product to be $1$ when $S \ssm \set{l} = \emptyset$.   This is well defined because $2, 3, \ldots, n-1$ are invertible.   Then, by induction on $n$, 
 $$\prod_{l \in S} (l+x)^{-1}  \ = \ \sum_{l \in S}  \lambda_{l} (l+x)^{-1}$$ 
  in $A_n(R)$, the  crucial calculation for the induction step being that 
	$$(l+x)^{-1} (m +x)^{-1} \ = \ (m-l)^{-1} (l+x)^{-1} + (l- m)^{-1} (m+x)^{-1}$$ for all  $m\in\{1, 2, \ldots, n-1\}$ and $l \in  \set{0, 1, \ldots, m-1}$.   
 So  $\prod_{l \in S} (l+x)^{-1}$ is in the span. 

Next consider $x^{h_0}(1+x)^{h_1} \cdots (n-1+x)^{h_{n-1}}$ where each $h_i$ is a non-positive integer.  We show it too is in the span by inducting on $\sum_{i=0}^{n-1} \abs{h_i}$.  The base case is immediate and the previous paragraph gives the induction step:   let $S = \set{ i \mid h_i <0 }$ and  $$\varepsilon_i  \ = \ \begin{cases}  1 &  \text{ if }   i \in S  \\  0 &  \text{ if }  i \notin S. \end{cases}$$ for each $i$, then  
\begin{align*}
x^{h_0}(1+x)^{h_1} & \cdots (n-1+x)^{h_{n-1}} \\ 
 & =  \ \left(\rule{0pt}{10pt}x^{h_0+ \varepsilon_0}(1+x)^{h_1 + \varepsilon_1} \cdots (n-1+x)^{h_{n-1}+ \varepsilon_{n-1}} \right) \prod_{l \in S} (l+x)^{-1} \\ 
 & =  \ \left(\rule{0pt}{10pt}x^{h_0 + \varepsilon_0}(1+x)^{h_1 + \varepsilon_1} \cdots (n-1+x)^{h_{n-1} + \varepsilon_{n-1}} \right) \sum_{l\in S} \lambda_{l} (l+x)^{-1}.  
\end{align*}

To complete the proof that the given set spans it is enough to show that $p(x)(m+x)^{-k}$ is in the span whenever $p(x) \in R[x]$,   $m\in \set{0, \ldots, n-1}$, and $k > 0$.  After all, any element of $A_n(R)$ is an $R$-linear combination of products of powers of $x, (1+x), \ldots, (n-1+x)$ and so by the previous result is an $R$-linear combination of some such $p(x)(m+x)^{-k}$.  Well, write $p(x) = (m+x) q(x) +s$ for some $q(x) \in R[x]$ and $s \in R$.  Then  
\begin{align*}
p(x)(m+x)^{-k} &  
  \ = \   \left(\rule{0pt}{10pt} q(x) +s(m+x)^{-1} \right)(m+x)^{-k +1}
\end{align*} 
which by induction on $\abs{k}$ is in the span.  

For linear independence, suppose $$0 \ = \ \sum_{j =0}^{d_{\infty}} \mu_j x^j + \sum_{i=0}^{n-1} \sum_{j=1}^{d_i} \lambda_{i,j} (i+x)^{-j}$$ in $A_n(R)$  for some $\mu_j, \lambda_{i, j} \in R$. Multiplying through by $x^{d_0}(1+x)^{d_1}\cdots (n-1+x)^{d_{n-1}}$   and comparing coefficients we see that $0=\mu_0=\mu_1=\cdots = \mu_{d_\infty}$. The constant term on the right hand side is $\lambda_{0, d_0}\cdot 1^{d_1}\cdot 2^{d_2}\cdot\ldots\cdot (n-1)^{d_{n-1}}$. As $2, \ldots, (n-1)$ are invertible in $R$, we must have  $\lambda_{0, d_0}=0$.
Repeatedly  dividing through by $x$ and analyzing the constant term gives $\lambda_{0, j}=0$ for all $j$. Viewing the resulting polynomial as a polynomial in $x-1$ rather than $x$ and applying the same technique yields $\lambda_{1, j}=0$ for all $j$. Then viewing it as a polynomial in $x-2$, then $x-3$, and so on, gives $\lambda_{i, j}=0$ for all $i,j$.
\end{proof}

In the light of this lemma we will, in the remainder of this section and the next talk about \emph{the $(\ast +x)^{-j}$ or the   $x^{j}$ coefficient of a $p \in A_n(R)$}, meaning the coefficient of that term when $p$ is expressed as a linear combination of    the basis established in Lemma~\ref{basis lemma}. 

   \begin{lemma} \label{First bijection lemma}  
 Suppose $\ast \in \set{0, 1, \ldots, n-1}$ and $q_0, \ldots, q_{n-1} \in \Z$, and $q_{\ast} =0$.  
  Given $\lambda_{\ast,1}, \lambda_{\ast,2}, \ldots$ in $R$, all but finitely many of which are zero, take $p$ to be any element of $A_n(R)$ such that the 
 coefficients of $(\ast + x)^{-1}, (\ast + x)^{-2}, \ldots$ are   $\lambda_{\ast,1}, \lambda_{\ast,2}, \ldots$.   Let  $\lambda'_{\ast,1}, \lambda'_{\ast,2}, \ldots$ be the coefficients of $(\ast + x)^{-1}, (\ast + x)^{-2}, \ldots$  in  $$p'  \  := \  x^{q_0} (1+x)^{q_1} \cdots (n-1+x)^{q_{n-1}} p.$$   
 Then  $\lambda'_{\ast,1}, \lambda'_{\ast,2}, \ldots$ depend only on  $\lambda_{\ast,1}, \lambda_{\ast,2}, \ldots$ and    $$(  \lambda_{\ast,1},  \ \lambda_{\ast,2}, \ \ldots  )   \ \mapsto \  (  \lambda'_{\ast,1}, \ \lambda'_{\ast,2}, \  \ldots   )$$ is a   bijection from the set of finitely supported sequences of elements of $R$ to itself.   Moreover, if
 $0= \lambda_{\ast,2} =  \lambda_{\ast,3} =  \cdots$, then
 \begin{align*}
  ( \lambda'_{\ast,1},  \  \lambda'_{\ast,2},   \ \lambda'_{\ast,3}, \  \ldots) = \left(  \lambda_{\ast,1} \prod_{i \in \set{0, \ldots, n-1 } \ssm \set{\ast} } (i-\ast)^{q_i}, \  0, \  0,  \ \ldots \right).
  \end{align*}
 
\end{lemma}  
  
 \begin{proof}
 It is enough to prove this in the special case $p'=(i+x)p$ where one of $q_0, \ldots, q_{n-1}$, denoted $q_i$, is $1$ and all others are $0$, for a general instance can  be reached by composing a suitable sequences of instances of this special case (and its `inverse').   Note that $i\neq \ast$, and so  we will be able to invert $(i - \ast)$.

 Express 
 \begin{align}
 p \  = \  & \sum_{j =0}^{\infty} \mu_j   x^j + \sum_{l=0}^{n-1} \sum_{j=1}^{\infty} \lambda_{l, j} (l+x)^{-j}, \label{p} \\
 p' \ = \  &  \sum_{j =0}^{\infty} \mu'_j x^j + \sum_{l=0}^{n-1} \sum_{j=1}^{\infty} \lambda'_{l,j} (l+x)^{-j} \label{p'}
 \end{align}
  where each $\mu_j, \mu'_j, \lambda_{l, j}, \lambda'_{l, j} \in R$ (and only finitely many are non-zero)---that is,  as  linear combinations of the basis established in Lemma~\ref{basis lemma}.  We prove the special case by calculating $(\mu'_0, \mu'_1, \ldots )$  and $(\lambda'_{l,1}, \lambda'_{l,2}, \ldots )$.  
  
 For $i, l \in \set{0, \ldots, n-1}$, 
$$(i +x)  \sum_{j =0}^{\infty} \mu_j x^j  \ =  \    i \mu_0   + (\mu_0+i \mu_1) x^1 + (\mu_1+i \mu_2) x^2 + \cdots$$ 
and, as $(i+x)(l+x)^{-j}  \ = \    (l+x)^{-j+1} + (i - l)(l+x)^{-j}$,
\begin{align}
(i +x)  \sum_{j=1}^{\infty} \lambda_{l,j} (l+x)^{-j}   \ = \ \  &  \sum_{j=1}^{\infty} \lambda_{l,j} (l+x)^{-j+1} +  \sum_{j=1}^{\infty} \lambda_{l,j}  (i - l) (l+x)^{-j} \nonumber \\
 \ = \ \  &  \lambda_{l,1} +     \sum_{j=1}^{\infty} \left(   \lambda_{l,j+1}    +    \lambda_{l,j}(i - l)   \right) (l+x)^{-j}.  \label{use}
\end{align}
So 
 $$(\lambda'_{\ast,1}, \lambda'_{\ast,2}, \ldots)  \  =  \ \left( \lambda_{\ast,2}    +   \lambda_{\ast,1} (i - \ast),  \ \lambda_{\ast,3}    +    \lambda_{\ast,2} (i - \ast), \  \ldots \right),$$
and evidently the only coefficients from \eqref{p}  this depends on are $\lambda_{\ast,1}, \lambda_{\ast,2}, \ldots$.  
Also we find that if  
 $0= \lambda_{\ast,2} =  \lambda_{\ast,3} =  \cdots$, then
 \begin{align*}
  ( \lambda'_{\ast,1}, \ \lambda'_{\ast,2}, \ \lambda'_{\ast,3}, \ \ldots) = \left(  \lambda_{\ast,1}   (i-\ast), \ 0, \  0, \ \ldots \right),
  \end{align*}
which leads to the final claim.
  To see that 
$$(  \lambda_{\ast,1}, \ \lambda_{\ast,2}, \  \ldots  )   \ \mapsto \  (  \lambda'_{\ast,1}, \  \lambda'_{\ast,2}, \ \ldots   )$$
is invertible when $i \neq \ast$, consider any $m$ such that  $ \lambda'_{\ast, q} =0$ for all $q>m$.  Then $0= \lambda_{\ast, m+1} = \lambda_{\ast, m+2}= \cdots$ as otherwise the sequence  $\lambda_{\ast,1}, \lambda_{\ast,2}, \ldots$ would not be finitely supported.  And 
\begin{align*}
\lambda_{\ast,m}  \ = \  \ &  (i-\ast)^{-1} \lambda'_{\ast,m}  \\
\lambda_{\ast,m-1}  \ = \ \ & (i-\ast)^{-1} (\lambda'_{\ast,m-1} - \lambda_{\ast, m}) \\
 \vdots \ \ \ \\
\lambda_{\ast,1}  \ = \  \ & (i-\ast)^{-1}(\lambda'_{\ast,1} - \lambda_{\ast, 2}). 
\end{align*}
\end{proof}

 \begin{lemma} \label{Second bijection lemma}  
 Suppose $q_0, \ldots, q_{n-1} \in \Z$ and $\sum_i q_i=0$.  
  Given $\mu_0, \mu_1, \ldots$ in $R$, all but finitely many of which are zero, take $p$ to be any element of $A_n(R)$ such that  the 
 coefficients of $x^0, x^1, \ldots$ are   $\mu_0, \mu_1, \ldots$.   Let  $\mu'_0, \mu'_1, \ldots$ be the coefficients of $x^0, x^1, \ldots$  in         $$p'  \  := \  x^{q_0} (1+x)^{q_1} \cdots (n-1+x)^{q_{n-1}} p.$$   
 Then $\mu'_0, \mu'_1, \ldots$ depend only on $\mu_0, \mu_1, \ldots$ and  
 $$(\mu_0, \ \mu_1, \ \ldots )   \ \mapsto \  (\mu'_0,  \ \mu'_1, \ \ldots )$$
is a   bijection from the set of finitely supported sequences of elements of $R$ to itself.      Moreover, if $0= \mu_1 =  \mu_2 = \cdots$, then   $(\mu'_0,   \mu'_1, \mu'_2, \ldots) = (\mu_0,  0, 0, \ldots)$.
\end{lemma}  
  
\begin{proof}
We follow a similar approach to our proof of Lemma~\ref{First bijection lemma}.  This time, as $\sum_i q_i=0$, it is  enough to prove the result in the special case $p'=x^{-1}(i+x)p$ where $q_0 =-1$, $q_i =1$ and all $q_j=0$ for all $j \neq 0,i$.

Again, consider $p$ and $p'$ expressed as in \eqref{p} and \eqref{p'}.
The crucial calculations this time are that 
$$x^{-1}(i +x)  \sum_{j =0}^{\infty} \mu_j x^j  \ =  \    i \mu_0x^{-1}   + (\mu_0+i \mu_1) x^0 + (\mu_1+i \mu_2) x^1 + \cdots$$
and for $l \in \set{0, 1,  \ldots, n-1}$, using \eqref{use},
\begin{align*}
x^{-1}(i +x)  \sum_{j=1}^{\infty} \lambda_{l,j} (l+x)^{-j}   \ = \ \  &   \lambda_{l,1} x^{-1} +     \sum_{j=1}^{\infty} \left(   \lambda_{l,j+1}    +   (i - l) \lambda_{l,j}  \right) x^{-1} (l+x)^{-j} 
\end{align*}
which  has no $x^0, x^1, \ldots$ terms when written as a linear combination of the basis elements since, by induction on $j$  and when $l\neq 0$,  $$x^{-1} (l+x)^{-j}  \ = \  l^{-j} x^{-1} - l^{-j}(l+x)^{-1} - l^{-j+1}(l+x)^{-2} -  \cdots -  l^{-1} (l+x)^{-j}.$$

So 
$$(\mu'_0, \ \mu'_1,  \ \ldots  )  \ = \  (     \mu_0 + i \mu_1,  \ \mu_1 + i \mu_2, \  \ldots  ),$$ 
and the final claim of the lemma is evident. To see that $$(\mu_0, \  \mu_1, \ \ldots )   \ \mapsto \  (\mu'_0,  \ \mu'_1,  \ \ldots )$$ is invertible, recall that $i\in\{1,2,\ldots, n-1\}$ (so $i$ is invertible), and consider any $m$ such that  $ \mu'_{q} =0$ for all $q>m$.  Then $0= \mu_{ m+1} = \mu_{ m+2}= \cdots$ as otherwise we would have $\mu_{q+1} = -i^{-1}\mu_q$ for all   $q>m$ and so  the sequence  $\mu_{0}, \mu_{1}, \mu_{2}, \ldots$ would not be finitely supported.  So 
\begin{align*}
\mu_{m}  \ = \  \ &  \mu'_{m}  \\
\mu_{m-1}  \ = \ \ &  \mu'_{m-1} - i\mu_{ m} \\
 \vdots \ \ \ \\
\mu_{0}  \ = \  \ & \mu'_{0} - i\mu_{1}. 
\end{align*}
\end{proof} 

\begin{cor} \label{coefficients cor}
If $k_\ast =0$, then the coefficients of $(\ast +x)^{-1},  (\ast +x)^{-2}, \ldots$  in  $x^{k_0} (1+x)^{k_1} \cdots (n-1+x)^{k_{n-1}}$ are all zero.     
\end{cor}

\begin{proof}
This is the final statement of Lemma~\ref{First bijection lemma} in the special case $p=1$  (and hence $\lambda_{\ast,j}=0$ for all $j$), and $q_l=k_l$ for all $l$.
\end{proof}

\begin{cor} \label{another coefficients cor}
If $k_\ast =-1$, then the coefficient  of $(\ast +x)^{-1}$  in  $x^{k_0} (1+x)^{k_1} \cdots (n-1+x)^{k_{n-1}}$ is 
$$\prod_{i \in \set{0, \ldots, n-1 } \ssm \set{\ast} } (i-\ast)^{k_i}.$$
\end{cor}

\begin{proof}
This is the final statement of Lemma~\ref{First bijection lemma} in the special case $p=(\ast +x)^{-1}$ (so $\lambda_{\ast,1}=1$ and $\lambda_{\ast,j}=0$ for all $j\neq 1$), $q_\ast=k_\ast+1=0$ and $q_l=k_l$ for all $l\neq \ast$.
\end{proof}

\begin{cor} \label{second coefficients cor}
If $k_{\infty} := - \sum_{i=0}^{n-1} k_i > 0$, then the coefficients of $x^0, x^1, \ldots$  in  $x^{k_0} (1+x)^{k_1} \cdots (n-1+x)^{k_{n-1}}$ are all zero.     
\end{cor}

\begin{proof}
This is the final statement of Lemma~\ref{Second bijection lemma}  with $q_0 = k_0 + k_{\infty}$ and $q_i = k_i$ for all other $i$ (so $\sum_{i=0}^{n-1} q_i=0$ as required) in the special case $p=x^{-k_{\infty}}$  (and since $k_{\infty}>0$, we have $\mu_j=0$ for all $j$).
\end{proof}
 
\begin{cor} \label{third coefficients cor}
If $\sum_{i=0}^{n-1} k_i = 0$, then  in  $x^{-k_0} (1+x)^{-k_1} \cdots (n-1+x)^{-k_{n-1}}$ 
 the coefficient  of $x^0$ is $1$ and  the coefficients of $x^1, x^2, \ldots$ are all zero.     
\end{cor}

\begin{proof}
This is the final statement of Lemma~\ref{Second bijection lemma} in the special case $p=1$  (so $\mu_0=1$ and $\mu_j=0$ for all $j\neq 0$) and $q_i=-k_i$ for all $i$.
\end{proof} 
 
 \begin{lemma} \label{shift coefficients}
For $p \in A_n(R)$, 
\begin{enumerate}
\item the coefficients  of $x^{0}, x^1, \ldots$ in $p$ equal those of $x^1, x^2, \ldots$ in $xp$, \label{case i}
\item  the coefficients  of $(\ast+x)^{-1}, (\ast+x)^{-2}, \ldots$ in $(\ast+x)p$ equal those of $(\ast+x)^{-2}, (\ast+x)^{-3}, \ldots$ in $p$.  
\end{enumerate}  
 \end{lemma}
 
\begin{proof}
Calculate in the manner of our proof of Lemma~\ref{First bijection lemma}.  The crucial point for (\textit{i}) is that $x(l+x)^{-j}  \ = \    (l+x)^{-j+1}  - l (l+x)^{-j}$ has no  $x^1, x^2, \ldots$  terms when $j \geq 1$.  The crucial points for  (\textit{ii}) are that  $(\ast+x)(l+x)^{-i}  \ = \    (l+x)^{-i+1} + (\ast - l)(l+x)^{-i}$  and $(\ast +x) x^j = \ast x^j + x^{j+1}$ have no  $(\ast+x)^{-1}, (\ast+x)^{-2}, \ldots$ terms when $l \in \set{0, 1, \ldots, n-1} \ssm \set{\ast}$ and $i \geq 1$ and when $j \geq 0$.
\end{proof}

\subsection{The bijection $\Phi$ between  $\Gamma_n(R)$ and the vertices of $\mathcal{H}_n(R)$} \label{general bijection}

 Define a map $\Phi$   from $\Gamma_n(R) = A_n(R) \rtimes \Z^n$ to the vertices of $\mathcal{H}_n(R)$ by
 $$(f, (h_0, \ldots, h_{n-1})) \ \mapsto \ ((\mathbf{a}^{\infty},h_{\infty}), (\mathbf{a}^0, h_0), \ldots,  (\mathbf{a}^{n-1},h_{n-1}) )$$ 
 where $h_{\infty} := -  h_0 -  \cdots - h_{n-1}$ and the sequences  $\mathbf{a}^{\infty},  \mathbf{a}^{0}, \ldots,  \mathbf{a}^{n-1}$  will be defined as follows (guided by Remark~\ref{rank-2 polynomial remark}). They  list the  coefficients of  elements of  $A_n(R)$, expressed as linear combinations of  the basis from Lemma~\ref{basis lemma}, specifically, for  $\ast=0, \ldots, n-1$,
 \begin{itemize}
 \item $\mathbf{a}^{\infty}$ lists the coefficients of $x^{0}, x^{1}, \ldots$ in $x^{h_{\infty}}f$, and 
 \item    $\mathbf{a}^{\ast}$ lists the coefficients of   $(\ast+x)^{-1}, (\ast+x)^{-2}, \ldots$ in $(\ast+x)^{-h_{\ast}}f$.  
 \end{itemize}
 
Our proof that $\Phi$ is a bijection will  involve 
\begin{equation*}
 \hat{f} \  := \  x^{-h_0} (1+x)^{-h_1} \cdots (n-1+x)^{-h_{n-1}}  f \label{fhat}
 \end{equation*}
and further sequences $\mathbf{b}^{\infty},  \mathbf{b}^{0}, \ldots,  \mathbf{b}^{n-1}$  defined by: 
 \begin{itemize}
\item $\mathbf{b}^{\infty}$ lists the coefficients of $x^{0}, x^{1}, \ldots$ in  $\hat{f}$, and
 \item  $\mathbf{b}^{\ast}$ lists the coefficients of   $(\ast+x)^{-1}, (\ast+x)^{-2}, \ldots$ in $\hat{f}$.
 \end{itemize}

 \begin{prop}  \label{it's a bijection}
  $\Phi$ is a bijection.
 \end{prop} 

\begin{proof} 
Suppose  $\mathbf{v} = ((\mathbf{a}^{\infty},h_{\infty}), (\mathbf{a}^0, h_0), \ldots,  (\mathbf{a}^{n-1},h_{n-1}))$ is a vertex of $\mathcal{H}_n(R)$ and so  $h_{\infty} = -  h_0 -  \cdots - h_{n-1}$.  We will explain that there is a unique  $g=(f, (h_0, \ldots, h_{n-1}))$ with $\Phi(g) = \mathbf{v}$.

The idea is to find the sequences $\mathbf{b}^{\infty}$, $\mathbf{b}^{0}$, \ldots, $\mathbf{b}^{n-1}$, for then we can recover $\hat{f}$ (and therefore $f$) from them since they  list all its coefficients when expressed as a linear combination of  the basis from Lemma~\ref{basis lemma}.   

For  $\ast = \infty$, this is possible (and unique) by   Lemma~\ref{Second bijection lemma}  applied
 with  $p = x^{h_{\infty}}f$ and $p' = \hat{f}$ (and so  $q_0 = -(h_{\infty}+ h_0)$, and $q_i=- h_i$ for $i=1, \ldots, n-1$).   It establishes a bijection taking $(\mu_0,  \mu_1,   \ldots)    =    \mathbf{a}^{\infty}$,  which lists the coefficients of $x^{0}, x^{1}, \ldots$ in  $x^{h_{\infty}}f$, to $\mathbf{b}^{\infty}   =    (\mu'_0, \mu'_1,  \ldots)$, which lists  the  coefficients of $x^{0}, x^{1}, \ldots$ in   $\hat{f}$.  Likewise, for $\ast = 0, 1, \ldots, n-1$, apply Lemma~\ref{First bijection lemma}  
with  $p = (\ast + x)^{-h_{\ast}}f$ and $p' = \hat{f}$  (and  so $q_i=-h_i$  for $i=0, 1, \ldots, n-1$ except that $q_{\ast} =0$).  It establishes a bijection taking  
$(\lambda_{\ast,1},   \lambda_{\ast,2},    \ldots )   =   \mathbf{a}^{\ast}$,  
which lists the coefficients of  $(\ast+x)^{-1}, (\ast+x)^{-2}, \ldots$
in $(\ast + x)^{-h_{\ast}}f$, to  $\mathbf{b}^{\ast}   =   (\lambda'_{\ast,1},   \lambda'_{\ast,2},   \ldots )$, which  lists the coefficients of  $(\ast+x)^{-1}, (\ast+x)^{-2}, \ldots$ in  $\hat{f}$.
\end{proof}

\subsection{Extending $\Phi$}

Next we  show  that $\Phi$ extends to a graph-isomorphism from the Cayley graph  $\mathcal{C}$ of $\Gamma_n(R)$ with respect to the generating set $$\set{ \, \left. (r, \e_i), \ (r, \e_j)(r, \e_k)^{-1}  \, \right| \, r \in R, \  0 \leq i, j, k  \leq  n-1  \textup{ and }   j < k \, }$$   
 to the 1-skeleton of $\mathcal{H}_n(R)$.  

Recall that we denote the  standard basis for  $\Z^n$ by  $\mathbf{e}_0, \ldots, \mathbf{e}_{n-1}$.
So, if  $\mathbf{h} = (h_0, \ldots, h_{n-1}) \in \Z^n$, then  $\mathbf{h} + \mathbf{e}_i \ = \  (h_0, \ldots, h_{i-1}, h_i +1, h_{i+1}, \ldots,  h_{n-1})$.  Recall that for such $\mathbf{h}$ and for $f \in A_n(R)$, $$f \cdot \mathbf{h} \ = \  f x^{h_0} (1+x)^{h_1} \cdots (n-1+x)^{h_{n-1}}.$$  
(Warning: $f \cdot \mathbf{0} =f$ and  $f \cdot (\mathbf{h} + \mathbf{h}')$ equals $(f \cdot \mathbf{h}) \cdot \mathbf{h}'$, and not in general  $f\cdot \mathbf{h} + f \cdot \mathbf{h}'$.)  Also recall that the group operation on $\Gamma_n(R)$ is   $(f,\mathbf{h})(\hat{f},\hat{\mathbf{h}})   =   (f + \hat{f} \cdot {\mathbf{h}},\mathbf{h} + \hat{\mathbf{h}})$.   

Suppose $g   =   (f, \mathbf{h} ) \in \Gamma_n(R)$ where $f \in A_n(R)$ and $\mathbf{h}   \in \Z^n$. 
We  show below  that post-multiplying $g$ by the elements of the generating set  and their inverses gives

\begin{align}
g(r, \mathbf{e}_j) \ = \ & \left( f+ r \cdot \mathbf{h},  \ \mathbf{h} + \mathbf{e}_j \right), \label{j} \\
g(r, \mathbf{e}_j)^{-1} \ = \ &   \left( f -r \cdot  (\mathbf{h}- \mathbf{e}_j),  \ \mathbf{h} - \mathbf{e}_j \right),     \label{-j}  \\
 g (r, \e_j)(r, \e_k)^{-1}  \ = \  &    \left( f +  (k-j) r \cdot (\mathbf{h} - \mathbf{e}_k ), \ \mathbf{h} +  \mathbf{e}_j -\mathbf{e}_k \right),   \label{j-k} \\ 
  g  (r, \e_k)(r, \e_j)^{-1}   \ = \ &     \left( f +  (j-k) r \cdot (\mathbf{h} - \mathbf{e}_j ), \ \mathbf{h} +  \mathbf{e}_k -\mathbf{e}_j \right)      \label{k-j}
 \end{align}
for all $r \in R$ and all $j,k\in\{0, \ldots, n-1\}$.
The explanation is that \eqref{j} is immediate from how group multiplication is defined,  \eqref{-j} uses that
 $$(r, \mathbf{e}_j)^{-1}   \ =  \   ( -r \cdot(-\mathbf{e}_j), - \mathbf{e}_j),$$ the key calculation for \eqref{j-k} is that $$ r\cdot \mathbf{h} - r \cdot (\mathbf{h} + \mathbf{e}_j - \mathbf{e}_k ) \ = \ r  \left( 1- \frac{j+x}{k+x} \right) \cdot  \mathbf{h} \ = \ r  \frac{k-j}{k+x}  \cdot  \mathbf{h}     \  = \ (k-j) r \cdot (\mathbf{h} - \mathbf{e}_k ),$$ and \eqref{k-j} is     immediate from  \eqref{j-k}.

Suppose$$\Phi(g) \ =  \ ((\mathbf{a}^{\infty},h_{\infty}), (\mathbf{a}^0, h_0), \ldots,  (\mathbf{a}^{n-1},h_{n-1}) ).$$ 
We claim next that     $\Phi$ maps
 \begin{align*}
 g(r, \mathbf{e}_j) \ \mapsto \ & \left(   \,    \left(  \rule{0pt}{13pt} \left(a_2^{\infty}, a_3^{\infty}, \ldots \right), h_{\infty}-1 \right),  \  \ldots,   \left( \rule{0pt}{13pt}  \left( \alpha_j  + r \beta_j, a_1^j, a_2^j, \ldots \right), h_j+1 \right),    \    \ldots  \,     \rule{0pt}{16pt}\right), \\
g(r, \mathbf{e}_j)^{-1} \ \mapsto \ & \left(   \,    \left(  \rule{0pt}{13pt} \left(  \alpha'_j   - r \beta'_j, a_1^{\infty}, a_2^{\infty}, \ldots \right), h_{\infty}+1 \right),  \  \ldots,      \left( \rule{0pt}{13pt}  \left(a_2^j, a_3^j, \ldots \right), h_j-1 \right),    \ \ldots       \rule{0pt}{16pt}\right), \\
g (r, \e_j)(r, \e_k)^{-1} \ \mapsto \ & \left(    \, \ldots,   \,  \left(  \rule{0pt}{13pt}  \left(   \alpha_{jk}  + r \beta_{jk}, a^j_1, a^j_2, \ldots \right), h_j+1\right),   \ \ldots, \left(  \rule{0pt}{13pt} \left(a_2^{k}, a_3^{k}, \ldots \right), h_{k}-1 \right),  \, \ldots    \rule{0pt}{16pt}\right),   \\
g(r, \e_k)(r, \e_j)^{-1}  \ \mapsto \ & \left(    \, \ldots,   \left(  \rule{0pt}{13pt} \left(a_2^{j}, a_3^{j}, \ldots \right), h_{j}-1 \right),  \,   \ \ldots,   \left(  \rule{0pt}{13pt}  \left(  \alpha'_{jk}  +  r \beta'_{jk}, a^k_1, a^k_2, \ldots \right), h_k+1\right),    \, \ldots    \rule{0pt}{16pt}\right),  
\end{align*}
where  the pairs indicated by ellipses are unchanged from the corresponding $(\mathbf{a}^i, h_i)$   in $\Phi(g)$,  and in terms of linear combinations of the basis  established in Lemma~\ref{basis lemma},
\begin{align*}
& \alpha_j  \text{ is the coefficient of }  (j+x)^{-1}  \text{ in }  (j+x)^{-h_{j}-1} f,  \\
& \alpha'_j     \text{ is the coefficient of }   x^0  \text{ in }     x^{h_{\infty}+1} f,     \\
& \alpha_{jk}    \text{ is the coefficient of }    (j+x)^{-1}  \text{ in }  (j+x)^{-h_{j}-1} f,    \\
& \alpha'_{jk}    \text{ is the coefficient of }  (k+x)^{-1}   \text{ in }   (k+x)^{-h_{k}-1} f,   \\
& \beta_j  = \prod_{i \in \set{0, \ldots, n-1 } \ssm \set{j}}  (i-j)^{h_i}, \text{the coefficient of }   (j+x)^{-1}  \text{ in }   (j+x)^{-h_{j}-1} \cdot \mathbf{h},\\
& \beta'_j = 1,    \text{the coefficient of }   x^0  \text{ in }   x^{h_{\infty}+1} \cdot (\mathbf{h} - \mathbf{e}_j), \\
& \beta_{jk} =  \prod_{i \in \set{0, \ldots, n-1 } \ssm \set{j} }  (i-j)^{h_i},  \text{the coefficient of }  (j+x)^{-1}   \text{ in }  (k-j) (j+x)^{-h_{j}-1} \cdot (\mathbf{h}-\mathbf{e}_k),   \\
& \beta'_{jk} =   \prod_{i \in \set{0, \ldots, n-1 } \ssm \set{k} } (i-k)^{h_i},   \text{the coefficient of }  (k+x)^{-1}  \text{ in }   (j-k)(k  +x)^{-h_{k}-1} \cdot (\mathbf{h}-\mathbf{e}_j).
\end{align*}

(The values of the coefficients   $\beta_j, \beta_{jk}$ and $\beta'_{jk}$  are as stated as a consequence of Corollary~\ref{another coefficients cor}  and $\beta'_j$ as a consequence of Corollary~\ref{third coefficients cor}.)

Here is why.   First note that the second entries (those involving $h_\infty, h_1, \ldots, h_{n-1}$) of all the coordinates are correct: they can be read off the vectors in the second coordinates of the righthand sides of \eqref{j}--\eqref{k-j}. 
Secondly, note that the case of  $\Phi(g(r, \mathbf{e}_j)(r, \mathbf{e}_k)^{-1})$ is identical to that of $\Phi(g(r, \mathbf{e}_k)(r, \mathbf{e}_j)^{-1})$, save that $j$ and $k$ are interchanged.  So we will only address the former.

Here is  why the  $(\mathbf{a}^i, h_i)$  indicated by ellipses in the above four equations are indeed the same as the corresponding  $(\mathbf{a}^i, h_i)$ in $\Phi(g)$. We compare the $(\ast+x)^{-1}, (\ast+x)^{-2}, \ldots$ coefficients of the appropriate polynomials.  

\textit{Case  $\Phi(g(r, \mathbf{e}_j))$.}  The polynomials in question are   $(\ast+x)^{-h_{\ast}}(f + r\cdot \mathbf{h})$ and $(\ast+x)^{-h_{\ast}}f$.  The relevant coefficients agree when $\ast \notin \{\infty,  j\}$ since those of   $$(\ast+x)^{-h_{\ast}}  r\cdot \mathbf{h} \ = \  r \prod_{l \in \set{0, 1, \ldots, n-1} \ssm \set{\ast}}  (l+x)^{h_l}$$ are all zero by Corollary~\ref{coefficients cor}. 

\textit{Case   $\Phi(g(r, \mathbf{e}_j)^{-1})$.} Similarly, the relevant coefficients of  
  $$(\ast+x)^{-h_{\ast}} (-r\cdot (\mathbf{h} - \mathbf{e}_j ))  \ = \  -r (j+x)^{-1} \prod_{l \in \set{0, 1, \ldots, n-1} \ssm \set{\ast}}  (l+x)^{h_l}$$  are all zero  when  $\ast \notin \{\infty,  j\}$ by the same corollary.
   
\textit{Case  $\Phi(g(r, \mathbf{e}_j)(r, \mathbf{e}_k)^{-1})$.}   Similarly, when  $\ast \notin \{\infty,  j, k\}$ the relevant coefficients of  $(\ast+x)^{-h_{\ast}} (k-j) r \cdot (\mathbf{h} - \mathbf{e}_k )$  are all zero.   And, for the $\ast= \infty$ case, the   coefficients of $x^0, x^1, \ldots$  in   $x^{h_{\infty} } (k-j) r \cdot (\mathbf{h} - \mathbf{e}_k )$  are all zero by 
Corollary~\ref{second coefficients cor}  (with $k_0=h_\infty+h_0$, $k_k=h_k-1$ and $k_l=h_l$ for all other $l$) since $h_{\infty} + h_0 + \cdots + h_{n-1} -1  = -1 <0$.

  Now we turn to the coordinates which differ after multiplication by a generator.  
 
\textit{Why the $\infty$-coordinate of $\Phi(g(r, \mathbf{e}_j))$ is $\left( \rule{0pt}{13pt} \left(a_2^{\infty}, a_3^{\infty}, \ldots \right), h_{\infty}-1 \right)$.}  
We need to determine the coefficients of   $x^0, x^1, \ldots$ in  $x^{h_{\infty} -1} (f+ r \cdot \mathbf{h})$. Those of     $x^{h_{\infty} -1} r \cdot \mathbf{h}$ are all zero 
by Corollary~\ref{second coefficients cor}.   Lemma~\ref{shift coefficients}(\textit{i}) tells us that the    coefficients of $x^0, x^1, \ldots$ in    $x^{h_{\infty} -1} f$ equal those of  $x^1, x^2, \ldots$ in  $x^{h_{\infty}} f$, and so  are  $a_2^{\infty}, a_3^{\infty}, \ldots$ by definition.

\textit{Why the $j$-coordinate of $\Phi(g(r, \mathbf{e}_j)^{-1})$ is   $\left( \rule{0pt}{13pt}  \left(a_2^j, a_3^j, \ldots \right), h_j-1 \right)$.} The  $(j+x)^{-1}, (j+x)^{-2}, \ldots$ coefficients of  $(j+x)^{-h_{j}+1} (f -  r\cdot (\mathbf{h} - \mathbf{e}_j))$  are  $a_2^j, a_3^j, \ldots$ since those of   $(j+x)^{-h_{j}+1}    r\cdot (\mathbf{h} - \mathbf{e}_j) = (j+x)^{-h_{j}}    r \cdot  \mathbf{h}  $ are all zero by Corollary~\ref{coefficients cor} and those of   $(j+x)^{-h_{j}+1} f$ equal the  $(j+x)^{-2}, (j+x)^{-3}, \ldots$ coefficients of  $(j+x)^{-h_{j}} f$ by Lemma~\ref{shift coefficients}(\textit{ii}).
 
 \textit{Why the  $k$-coordinate of $\Phi(g (r, \e_j)(r, \e_k)^{-1})$ is  $\left(  \rule{0pt}{13pt} \left(a_2^{k}, a_3^{k}, \ldots \right), h_{k}-1 \right)$.} The $(k+x)^{-1}, (k+x)^{-2}, \ldots$ coefficients of  $(k+x)^{-h_{k}+1} (f + (k-j) r\cdot (\mathbf{h} - \mathbf{e}_k))$ are $a_2^{k}, a_3^{k}, \ldots$ similarly to the previous case.

\textit{Why the $j$-coordinate of $\Phi(g(r, \mathbf{e}_j))$ is  $\left( \rule{0pt}{13pt}  \left( \alpha_j + r\beta_j, a_1^j, a_2^j, \ldots \right), h_j+1 \right)$.}  We need to check that the  $(j+x)^{-1}, (j+x)^{-2}, \ldots$ coefficients of   $(j+x)^{-h_{j}-1} (f+ r \cdot \mathbf{h})$ are $ \alpha_j + r\beta_j, a_1^j, a_2^j, \ldots$. The $(j+x)^{-2}, (j+x)^{-3}, \ldots$ coefficients are $a_1^j, a_2^j, \ldots$  since those of   $(j+x)^{-h_{j}-1}    r\cdot \mathbf{h} = (j+x)^{-1} ((j+x)^{-h_j}   r \cdot  \mathbf{h})  $ are all zero by Corollary~\ref{coefficients cor} and those of   $(j+x)^{-h_{j}-1} f$ equal the  $(j+x)^{-1}, (j+x)^{-2}, \ldots$ coefficients of  $(j+x)^{-h_{j}} f$ by Lemma~\ref{shift coefficients}(\textit{ii}) for the same reasons as in earlier cases.   Its $(j+x)^{-1}$-coefficient is  $\alpha_j +  r\beta_j$ by definition.

\textit{Why the $\infty$-coordinate of $\Phi(g(r, \mathbf{e}_j)^{-1})$ is  $\left(  \rule{0pt}{13pt} \left(  \alpha'_j  -  r \beta'_j, a_1^{\infty}, a_2^{\infty}, \ldots \right), h_{\infty}+1 \right)$.}  The $x^0, x^1, \dots$ coordinates of 
$x^{h_{\infty}+1} (f -  r\cdot (\mathbf{h} - \mathbf{e}_j))$ are  $ \alpha'_j  - r \beta'_j, a_1^{\infty}, a_2^{\infty}, \ldots$ for similar reasons.

\textit{Why the  $j$-coordinate of  $\Phi(g (r, \e_j)(r, \e_k)^{-1})$  is $\left(  \rule{0pt}{13pt}  \left(   \alpha_{jk}  + r \beta_{jk}, a^j_1, a^j_2, \ldots \right), h_j+1\right)$.}  The  $(j+x)^{-1}, (j+x)^{-2}, \ldots$ coefficients of  $(j+x)^{-h_{j}-1} (f + (k-j) r\cdot (\mathbf{h} - \mathbf{e}_k))$ are $\alpha_{jk}  + r \beta_{jk}, a^j_1, a^j_2, \ldots$ likewise.

The set of vertices $\mathcal{V}$ in $\mathcal{H}_n(R)$  that are reached by traveling from   $\Phi(g)$ along  a single edge partitions into $(n+1) n$ subsets:  travel along the unique downwards edge  in one of the $n+1$ coordinate-trees,  travel upwards  along one of an $R$-indexed family of edges in another, and    remain  stationary in the rest.

As we have seen, for each element $x$ of the generating set 
$$\set{ \, \left. (r, \e_i), \ (r, \e_j)(r, \e_k)^{-1}  \, \right| \, r \in R, \  0 \leq i, j, k  \leq  n-1  \textup{ and }   j < k \, }$$ 
the location of $\Phi(gx)$ and  $\Phi(gx^{-1})$ falls in one of these subsets.  Thereby the union of this generating set together with the set of the inverses of its elements has $(n+1) n$   subsets which correspond to the  $(n+1) n$ subsets of $\mathcal{V}$.    Indeed, each subset contains one $R$-indexed family of generators or inverse-generators.

 Since $\alpha_j$ and $\beta_j$ do not depend on $r$  and $\beta_j$ is invertible (since $2, 3, \ldots, n-1$ are invertible), for fixed $j$, the map  $r  \mapsto \alpha_j + r \beta_j$ is a bijection $R \to R$.  So $g (r, e_j)  \mapsto  \Phi(g(r, e_j))$ is a bijection between a subset of the neighbours of $g$   in the Cayley graph $\mathcal{C}$ and   one of these subsets of $\mathcal{V}$.  
 
Likewise, because $ \beta'_j,  \beta_{jk},  \beta'_{jk}$ are invertible (since $2, 3, \ldots, n-1$ are invertible),  
\begin{align*}
 r & \mapsto \alpha'_j - r \beta'_j, \\ 
r & \mapsto \alpha_{jk} +  r \beta_{jk},  \\ 
r  & \mapsto \alpha'_{jk} + r \beta'_{jk}
\end{align*}
 are all bijections $R \to R$.     So as  $\alpha'_j$, $\alpha_{jk}$, $\alpha'_{jk}$,  $\beta'_j$, $\beta_{jk}$, and $\beta'_{jk}$ do not depend on $r$, there are similar bijections between subsets of neighbours of $g$ and
 subsets of $\mathcal{V}$.  Combined, these  bijections give a bijection  from the neighbours of $g$ in $\mathcal{C}$ to the neighbours of $\Phi(g)$ in $\mathcal{V}$.
 
There are no double-edges and no edge-loops in either graph:  for the 1-skeleton of $\mathcal{H}_n(R)$ this is straightforward from the definition, and it therefore follows from the above for $\mathcal{C}$.   So $\Phi$ extends to an isomorphism from $\mathcal{C}$ to the  1-skeleton of $\mathcal{H}_n(R)$, completing our proof.

\section{Presentations of $\Gamma_1$,  $\Gamma_1(m)$,  $\Gamma_2$ and  $\Gamma_2(m)$}    \label{presentations section}

In this section we give presentations of $\Gamma_1$,  $\Gamma_1(m)$,  $\Gamma_2$ and  $\Gamma_2(m)$ which reflect their descriptions as horocyclic products of trees.  Our presentations for  $\Gamma_2$ include one which we will prove in Section~\ref{Cayley complex section} to have Cayley 2-complex $\mathcal{H}_2(\Z)$.

Our conventions are that $[a,b] = a^{-1}b^{-1} ab$ and $a^{nb} = b a^n b^{-1}$ 
 for group elements $a$, $b$ and integers $n$.  Our group actions are on the right. 

\begin{prop} \label{the presentations}
Presentations for the group $$\Gamma_1 \ = \     \Z \wr \Z   \  \cong \  \Z[x,x^{-1}]\rtimes \Z  \   =  \  \set{  \ \left. \left(\begin{array}{cc}x^{k}  & f \\0 & 1\end{array}\right) \  \right| \  k  \in \Z, \ f \in \Z[x,x^{-1}] \  }$$  include 
\begin{enumerate}[label=(\roman*)]
\item $\left\langle  \ a, t \ \left| \   \left[a,a^{t^k}\right]=1 \ (k \in \Z)  \  \right. \right\rangle$, \label{i}
\item  $\left\langle  \ \lambda, \mu \ \left| \   \lambda^k \left(\lambda^{-1} \mu \lambda^{-1}\right)^k = \mu^k \lambda^{-k} \ (k \in \Z)  \
 \right. \right\rangle$,  \label{ii}
\item     $\left\langle  \ \lambda_i \  (i \in \Z) \ \left| \  {\lambda_i}^k{\lambda_j}^{-k} =  {\lambda_{-j}}^k{\lambda_{-i}}^{-k}  \ (i, j, k \in \Z)  \  \right. \right\rangle.$ \label{iii}
\end{enumerate}
These are related  via $\lambda=t$, $\mu=at$, and $\lambda_i = a^it$.  
   \end{prop}

\begin{proof}   
As an abelian group, $$\Z[x,x^{-1}]  \ =  \  \bigoplus_{i \in \Z} \Z \ = \  \langle \,  a_i \  (i \in \Z)  \mid  [a_i, a_j] =1 \ \forall i,j \, \rangle.$$   
So   $\Z[x,x^{-1}]\rtimes \Z = \  \langle \,   t, a_i \  (i \in \Z)  \mid t a_i t^{-1} = a_{i+1}, \  [a_i, a_j] =1 \ \forall i,j \, \rangle$, which simplifies with $a = a_0$ to give   \emph{\ref{i}}.  
 
For  \emph{\ref{ii}}, it suffices to show that
$\left\langle  \ a, t \ \left| \   \left[a,a^{t^k}\right]=1 \ (k \in \Z)  \
 \right. \right\rangle$
can be re--expressed as   
$$\left\langle  \ a, t \ \left| \  t^k(t^{-1}a)^{k} = (at)^k t^{-k}  \ (k \in \Z)  \
 \right. \right\rangle,$$
 since the latter becomes  \emph{\ref{ii}}   via $\lambda=t$ and $\mu=at$.   
 Well, $t^k(t^{-1}a)^{k}$ and $(at)^k t^{-k}$ freely equal
$(t^{k-1}at^{-(k-1)})  \ldots   (tat^{-1}) \, a$ and $a \, (tat^{-1}) \ldots (t^{k-1}at^{-(k-1)})$,   
respectively, and a straight--forward induction shows that the family $\set{a^{t^k}a = a a^{t^k}}_{k \in \Z}$  is equivalent to  $$\set{ \, a^{t^k} \cdots  a^t a = a a^t \cdots a^{t^k}, \ a^{t^{-k}} \cdots  a^{t^{-1}} a = a a^{t^{-1}} \cdots a^{t^{-k}} \, }_{k>0}.$$ 
 
Finally we establish \emph{\ref{iii}}.  If $\lambda_i = a^it$  then $\lambda_i$ must correspond to 
$ \left(\begin{array}{cc}x  &  i   \\0 & 1\end{array}\right)$ and so ${\lambda_i}^k$  to 
$\left(\begin{array}{cc}x^k &  i(1+ \cdots + x^{k-1})  \\0 & 1\end{array}\right)$ and  ${\lambda_i}^{-k}$ to  $\left(\begin{array}{cc}x^{-k} &  -i(x^{-k}+ \cdots + x^{-1})  \\0 & 1\end{array}\right)$.  From there it is easy to check that the relations ${\lambda_i}^k{\lambda_j}^{-k} =  {\lambda_{-j}}^k{\lambda_{-i}}^{-k}$ correspond to valid matrix identities, and so must be consequences of the relations $ \left[a,a^{t^k}\right]=1$  ($k \in \Z$).

Conversely, given that $\lambda_0 = \lambda = t$ and $\lambda_1=\mu=at$, we find that $\lambda_{-1} = a^{-1}t =  \lambda \mu^{-1} \lambda$, and so 
the relations $\lambda^k \left(\lambda^{-1} \mu \lambda^{-1}\right)^k = \mu^k \lambda^{-k}$ of  \emph{\ref{ii}}  are  ${\lambda_i}^k{\lambda_j}^{-k} =  {\lambda_{-j}}^k{\lambda_{-i}}^{-k}$  in the case $i=0$ and $j=-1$.
\end{proof}

On introducing  torsion,  adding the relation $a^m=1$ to   Presentation~\emph{\ref{i}} of Proposition~\ref{the presentations}, we get presentations for $\Gamma_1(m)$.  These can be reorganized in the manner of Presentations~\emph{\ref{ii}} and~\emph{\ref{iii}}, and in the case $m=2$ can be simplified significantly:

\begin{prop}
$$\begin{array}{rl}
\Gamma_1(2) & \!\!\! = \  (\Z/ 2\Z) \wr \Z   \ = \ \left\langle  \ \lambda, \mu \ \left| \   \left(\lambda^k \mu^{-k}\right)^2  = 1 \ (k \in \Z)  \
 \right. \right\rangle, \\
\Gamma_1(m) & \!\!\!  = \   (\Z/ m\Z) \wr \Z  \\  
& \!\!\!   = \ \left\langle  \ \lambda_0,   \ldots, \lambda_{m-1} \ \left| \  {\lambda_i}^k{\lambda_j}^{-k} =  {\lambda_{-j}}^k{\lambda_{-i}}^{-k}, \  \left({\lambda_i}^k{\lambda_j}^{-k}\right)^m=1 \ (i, j \in \Z/m\Z, \  k \in \Z)  \ \right. \right\rangle,
 \end{array}$$
where $m \geq 2$, $\lambda=t$, $\mu=at$, and $\lambda_i = a^it$.   
\end{prop}

\begin{proof}
The presentation for $\Gamma_1(2)$ comes from simplifying Presentation~\emph{\ref{ii}} of Proposition~\emph{\ref{the presentations}} using the relation $a^2=1$, which is equivalent to $\lambda^{-1}\mu\lambda^{-1}=\mu^{-1}$.  The family  $\lambda^k (\lambda^{-1} \mu \lambda^{-1})^k = \mu^k \lambda^{-k}$ becomes the family $(\lambda^k \mu^{-k})^2 = 1$.  The case $k=1$ provides the relation $a^2=1$.    

For   $\Gamma_1(m)$, consider adding the  family of relations $\left({\lambda_i}^k{\lambda_j}^{-k}\right)^m=1$ for all $i,j,k\in \Z$  to Presentation~\emph{\ref{iii}} of Proposition~\ref{the presentations}.  In particular this adds the relation $a^m=1$, which is the case:   $\left({\lambda_1}{\lambda_0}^{-1}\right)^m= \left( at \, t^{-1} \right)^m =   1$.  In the resulting group $\lambda_i = \lambda_j$ when $i =j$ modulo $m$ since then $a^it = a^jt$ because $a^m=1$.  This group must be $\Gamma_1(m)$ because all the remaining added relations hold in $\Gamma_1(m)$, after all when $k >0$ (and similarly when $k<0$), 
\begin{align*}
\left({\lambda_i}^k{\lambda_j}^{-k}\right)^m \ & = \ \left( (a^it)^k (a^j t)^{-k} \right)^m \\
&  = \  \left( a^i (t a^i t^{-1}) \cdots (t^{k-2} a^i t^{-(k-2)}) (t^{k-1} a^{i-j} t^{-(k-1)}) (t^{k-2} a^{-j} t^{-(k-2)}) \cdots (ta^{-j} t^{-1}) a^{-j} \right)^m \end{align*}
which is $1$ because $\left( a^{t^p} \right)^m =1$ and $a^{t^p}$ and $a^{t^q}$ commute in $\Gamma_1(m)$ for all $p,q \in \Z$.
\end{proof}

\begin{prop} \label{the presentations for gamma 2}
Presentations of     
$$\Gamma_2 \ = \ \set{  \ \left. \left(\begin{array}{cc}x^{k}(1+x)^{l}  & f \\0 & 1\end{array}\right) \  \right| \  k, l  \in \Z, \ f \in \Z\left[x,x^{-1}, (1+x)^{-1}\right] \  }$$  
 include 
\begin{enumerate}[label=(\roman*)]
\item $ \left\langle \ a,s,t \ \left| \   [a,a^t]=1, \ [s,t]=1, \ a^s=aa^t \ \right. \right\rangle$, 
 \label{I}
\item  $ \left\langle \ \mu,\nu,c,d \ \left| \   [\mu,\nu]=1, \ \mu^{-1}c^2\nu=c, \ \nu^{-1}d^2\mu=d \
\right. \right\rangle$,  \label{II}
\item    $ \left\langle \ \lambda_i, \mu_i, \nu_i \  (i \in \Z)   \ \left| \   \lambda_i = \nu_i \mu_i, \  \lambda_{i+j} = \mu_i\nu_j  \  (i,j \in \Z)  \ \right. \right\rangle.$ \label{III}
\end{enumerate}
These are related by  
$$a \mapsto     \left(\begin{array}{cc}1 & 1 \\0 & 1\end{array}\right), \qquad t  \mapsto \left(\begin{array}{cc} x & 0 \\0 & 1\end{array}\right), \qquad s  \mapsto \left(\begin{array}{cc} 1+ x & 0 \\0 & 1\end{array}\right),$$
$\mu=s$, $\nu=t^{-1}s$, $c=at$, and $d=t^{-1}a$, and 
  $\lambda_i = a^i t$ and $\mu_i=a^i s$ (and hence $\nu_i = \lambda_i {\mu_i}^{-1} = a^i t s^{-1} a^{-i}$).  
 \end{prop}

 The generators $\lambda_i$, $\mu_i$, and $\nu_i := \lambda_i {\mu_i}^{-1}$ agree with those employed in  Section~\ref{rank 2}. After all, 
 $$\lambda_i  \ = \  a^it \  \mapsto   \  \left(\begin{array}{cc}x & i \\0 & 1\end{array}\right) \qquad \text{and} \qquad \mu_i \ = \  a^is \  \mapsto  \   \left(\begin{array}{cc}1+x & i \\0 & 1\end{array}\right),$$ which 
   are alternative ways of expressing $(i, \e_0)$ and  $(i, \e_1)$.

 Presentation~\emph{\ref{I}} and the given matrix representation are due to Baumslag in \cite{Baumslag} and our proof below that they agree is an embellishment of the argument in his paper.   Presentation~\emph{\ref{II}} is striking as it shows that $\Gamma_2$ maps onto a free--product with amalgamation of two $\textup{BS}(1,2)$ groups (via identifying $\mu$ and $\nu$).   
   The generators of Presentation~\emph{\ref{III}} are those we used in Sections~\ref{rank 2} and \ref{rank n} to relate $\Gamma_2$ to a horocyclic product of trees. In Section~4 of \cite{GKKL}, presentations of similar matrix groups  are given (e.g.\ in Section~4.3.1) using techniques that are similar to those that follow and are based on ideas in \cite{Baumslag}.

In the course of proving Proposition \ref{the presentations for gamma 2} we will also establish:

\begin{lemma}[\emph{Normal form}] \label{normal form lemma} Elements  $g$ in   $\Gamma_2$,  presented as \textit{\ref{i}}, are represented by a unique word  
\begin{equation} \label{normal form word}
w_g \ =  \ {a^{m_1}}^{ t^{k_1}} \cdots {a^{m_K}}^{ t^{k_K}} \  {a^{n_1}}^{ s^{l_1}} \cdots {a^{n_{L}}}^{ s^{l_{L}}}   \  s^l t^k
\end{equation}
with $k_1, \ldots, k_K,  l_1, \ldots,  l_L, l,k, L, K \in \Z$ and  $m_1,  \ldots, m_K, n_1, \ldots, n_L \in \Z \ssm \set{0}$ satisfying  $k_1 <  \cdots < k_K$ and  $l_1 <  \ldots < l_L < 0$.     
\end{lemma}

\begin{proof}[Proof of  Proposition~\ref{the presentations for gamma 2} and Lemma~\ref{normal form lemma}.]
Let us establish the existence part of Lemma~\ref{normal form lemma}.  Suppose $w$ is any word on $a, s,t$ representing $g$.  
First convert $w$ to a word of the form $\prod_i a^{s^{p_i} t^{q_i}} s^l t^k$ by inserting suitable words on $\{s^{\pm 1}, t^{\pm 1}\}$   after each $a$ and then using the relation $[s,t]=1$.      
Then eliminate all the positive $p_i$ by expressing $a^{s^{p_i}}$ as a product of terms like $a^{t^j}$ using the relation $a^s = a a^t$.  
In $\Gamma_2$,  $[a, a^{t^n}] =1$ for all $n \geq 0$ as can be seen by an induction via $$1 \ = \ \left[a, a^{t^n}\right]^s \ = \ \left[a^s, (a^s)^{t^n}\right] \ = \  \left[aa^t, a^{t^n} a^{t^{n+1}}\right] \ = \ \left[a ,  a^{t^{n+1}}\right].$$ 
(We see here that the relation $a^s = a a^t$, which Baumslag calls \emph{mitosis}, is the key to coding the infinite family of defining relators  $\left[a,a^{t^n}\right]=1$ ($n \in \Z$) in a finite presentation.)
So, as $a^{s^i}$ can be expressed as a product of  terms of the form $a^{t^j}$ ($j \in \Z$), elements of the set  $\set{a^{s^i},  a^{t^j},   \mid i, j \in \Z }$ pairwise commute in $\Gamma_2$.   So we can rearrange terms to get the form of $w_g$.  
 
Next we observe that the map $\phi$ from the group presented by  \emph{\ref{i}} to the given matrix group, defined for $a$, $s$ and $t$ as indicated in the proposition,   is well-defined and is a homomorphism: the defining relations correspond to identities which hold in the matrix group.    It maps a group element $g$ represented by the word $w_g$ of Lemma~\ref{normal form lemma}   to 
 \begin{equation} \label{semi-direct}
 \left(\begin{array}{cc}1 &  f  \\0 & 1\end{array}\right) \left(\begin{array}{cc}x^k (1+x)^l & 0 \\0 & 1\end{array}\right)  \ =  \   \left(\begin{array}{cc}x^k (1+x)^l & f \\0 & 1\end{array}\right) ,
 \end{equation}  
 where $$f  \ =  \ m_1x^{k_1} + \cdots +  m_K x^{k_K} +  n_1(1+x)^{l_1} + \cdots +  n_L (1+x)^{l_L}.$$
So $\phi$ is surjective.  Now $\set{x^i, (1+x)^j \mid i,j \in \Z,  \ j <0 }$ is a basis for $\Z[x,x^{-1}, (1+x)^{-1}]$ as we saw in Section~\ref{Gamma_2 lamplighter}.  So $\phi$ is also injective and   the normal form words of Lemma~\ref{normal form lemma} each represent different group elements.  So  \emph{\ref{i}} is a presentation of $\Gamma_2$.

The translation between Presentations \emph{\ref{i}} and \emph{\ref{ii}} comes from that the relations $[s,t]=1$ and $[\mu,\nu]=1$ are  equivalent, and, in the presence of that commutator,  $\mu^{-1}c^2\nu=c$ and $\nu^{-1}d^2\mu=d$ are equivalent to $a^s=aa^t$ and $a^s=a^ta$, respectively.

Presentations \emph{\ref{i}} and \emph{\ref{iii}} agree as follows.   When $i=j=0$, the relation  $\lambda_{i+j} = \mu_i\nu_j$ becomes $[s,t]=1$, and, in the presence of $[s,t]=1$, when  $i=-j=1$,  it gives $a^s=a^t a$, and when    $-i=j=1$,  it gives $a^s=aa^t$.  Moreover, in terms of $a,s,t$ the relation $\lambda_{i+j} = \mu_i\nu_j$ is $a^{i+j}t = a^i s a^j t s^{-1} a^{-j}$, which holds in $\Gamma_2$ because  $a^i s a^j t s^{-1} a^{-j} t^{-1} a^{-i-j}   = a^i (s a  s^{-1})^j t a^{-j} t^{-1} a^{-i-j}  =  a^i (  a  a^t)^j  a^{-jt}   a^{-i-j}  =  a^{i+j}  a^{jt}  a^{-jt}   a^{-i-j}  = 1$.   
\end{proof}

The normal-form words of Lemma~\ref{normal form lemma} read off  lamplighter descriptions of group elements in which the configurations are supported on $L_{0,0}$ (that is, the $t$-axis and the negative half  of  the $s$-axis).   If a group  element $g$ positions the lamplighter far from $L_{0,0}$, then the    configuration supported on $L_{0,0}$ representing $g$ will differ dramatically from that representing $ga^{\pm 1}$, since the effect of propagating $\pm 1$ towards $L_{0,0}$ compounds in the manner of  Pascal's triangle.  
 
 A word on $a, s, t$ as per Presentation \emph{\ref{i}} for $\Gamma_2$ represents a group element whose   lamplighter description can be found as follows.  Start with the lamplighter located at $(0,0)$ and the configuration entirely zeroes.  Working through $w$ from left to right, increment the integer at the lamplighter's location by $\pm 1$ on reading an $a^{\pm 1}$, move the lamplighter one step to the right or left (the $t$-  or $t^{-1}$-direction) on reading a $t$ or $t^{-1}$, respectively, and move  the  lamplighter one step to the adjacent vertex in the $s$-  or $s^{-1}$-direction   on reading an $s$ or $s^{-1}$, respectively.
  
 For presentations of the groups $\Gamma_n(m)$ in general see Theorem~4.7 in \cite{BNW}.

\section{$\mathcal{H}_2(\Z)$ as a Cayley 2--complex}   \label{Cayley complex section}
  
In this section we  show that $\mathcal{H}_2(\Z)$ is the Cayley 2-complex of  
$$\Gamma_2 \ = \ \left\langle \ \lambda_i, \mu_i, \nu_i \  (i \in \Z)   \ \left| \   \lambda_i = \nu_i \mu_i, \  \lambda_{i+j} = \mu_i\nu_j  \  (i,j \in \Z)  \ \right. \right\rangle,$$ proving Theorem~\ref{Cayley complex}. 

Identify the Cayley graph (the 1-skeleton of the Cayley 2-complex)  with the 1-skeleton of  $\mathcal{H}_2(\Z)$  as per the  $n=2$ case of Theorem~\ref{main} (proved in Section~\ref{rank 2}).

First  we show that every $2$-cell in $\mathcal{H}_2(\Z)$ is bounded by an edge-loop which corresponds to a defining relation of $\Gamma_2$.
Suppose a point $\mathbf{p} = (p_0, p_1, p_2) \in  \mathcal{H}_2(\Z)$ is in the interior of a  2-cell $X$.  Then each $p_j$ is in the interior of an edge $I_j$ of the tree $\mathcal{T}_{\Z}$.   Let $\ell_j =   \min_{u \, \in \, I_j }  h(u)$ and  $x_j = h(p_j) - \ell_j$ for $j=0,1,2$.   It follows from  $h(p_0) + h(p_1) + h(p_2) = 0$  and $0<x_j<1$  that  $\ell_0 + \ell_1 + \ell_2$ is  either $-1$ or $-2$.  So $x_0 + x_1 + x_2$  is $1$ or $2$.  Say $X$ is of ``type 1'' or ``2'' accordingly.   Examples are shown in Figure~\ref{Cells fig} (with the vertices of the triangles  labeled by $(x_0, x_1, x_2)$-coordinates).

\begin{figure}[ht]
  \psfrag{S}{\small{Type $1$ example  with $(\ell_0, \ell_1, \ell_2) = (-1,0,0)$}}
    \psfrag{T}{\small{Type $2$ example with $(\ell_0, \ell_1, \ell_2) = (0,-1,-1)$}}
  \psfrag{p}{\small{$p$}}
  \psfrag{1}{\small{$1$}}
  \psfrag{0}{\small{$0$}}
    \psfrag{m}{\small{$-1$}}
  \psfrag{a}{\small{$p_0 =  - \frac{3}{4}$}}
  \psfrag{b}{\small{$p_1 =  \frac{1}{4}$}}
  \psfrag{c}{\small{$p_2 = \frac{1}{2}$}}
  \psfrag{x}{\small{$p_0 = \frac{3}{4} $}}
  \psfrag{y}{\small{$p_1=  -\frac{1}{4}$}}
  \psfrag{z}{\small{$p_2 = -\frac{1}{2}$}}
  \psfrag{i}{\small{$(1,0,0)$}}
  \psfrag{j}{\small{$(0,1,0)$}}
  \psfrag{k}{\small{$(0,0,1)$}}
    \psfrag{I}{\small{$(0,1,1)$}}
  \psfrag{J}{\small{$(1,0,1)$}}
  \psfrag{K}{\small{$(1,1,0)$}}
  \psfrag{L}{\small{$\lambda_{i+j}$}}
  \psfrag{M}{\small{$\mu_i$}}
  \psfrag{N}{\small{$\nu_j$}}
    \psfrag{B}{\small{$\nu_i$}}

  \psfrag{l}{$\lambda_{i}$}

 \centerline{\epsfig{file=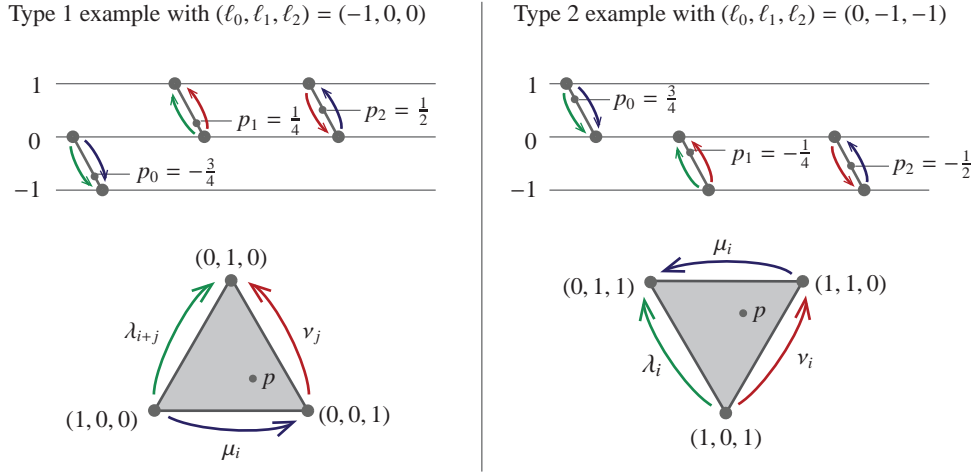}} \caption{Examples of 2-cells of type 1 and 2} \label{Cells fig}
\end{figure}

Consider moving  $\mathbf{p}$ within $X$ as parametrized by $(x_0,  x_1,  x_2)$.  It is on an edge in $\partial X$ when one of the $p_i$ is at an end of $I_i$ and is on a vertex when two (and hence all three) are at an end of $I_i$. So if $X$ is of type $1$, it   has  vertices, 
$(x_0,x_1,x_2)= (1,0,0)$, $(0,1,0)$, and  $(0,0,1)$, 
and  $\partial X$ is traversed by following the  edges  $(1-r,r,0)_{0 \leq r \leq 1}$, then $(0,1-r,r)_{0 \leq r \leq 1}$, and then $(r,0,1-r)_{0 \leq r \leq 1}$.    If $X$ is of type $2$, it   has  vertices, 
$(x_0,x_1,x_2)= (0,1,1)$, $(1,0,1)$, and  $(1,1,0)$, 
and  $\partial X$ is traversed by following  $(r,1-r,1)_{0 \leq r \leq 1}$, then $(1,r,1-r)_{0 \leq r \leq 1}$,  and then $(1-r,1,r)_{0 \leq r \leq 1}$.

Now, $\partial X$ corresponds to a length-$3$ relator in $\Gamma_2$, and matching the changes in heights as $\partial X$ is traversed with the height-changes indicated in the family of six displayed equations in our  proof of the $n=2$ case of Theorem~\ref{main} in Section~\ref{rank 2}, that relator must be $\lambda_k {\nu_j}^{-1}  {\mu_i}^{-1}$ for type 1, and  $\lambda_i^{-1} \nu_j  \mu_k$ for type 2,   for some $i, j, k \in \Z$.     

The workings of lamplighter model   illustrated in Figure~\ref{multiply figure}  allow us to see that  $\lambda_k \nu_j^{-1}  \mu_i^{-1}=1$ in $\Gamma_2$   if and only if $k-j-i=0$ since $\lambda_k \nu_j^{-1}  \mu_i^{-1}$ does not move the  lamplighter and  increments the lamp at the  lamplighter's location by  $k-j-i$.  That is, the relation is $\lambda_{i+j} = \mu_i \nu_j$ for some $i,j \in \Z$.      Similarly,  $\lambda_i^{-1} \nu_j  \mu_k =1$ in $\Gamma_2$   if and only if $i=j=k$ since $\lambda_i^{-1} \nu_j  \mu_k$ does not move the  lamplighter and  transforms a triangle of numbers $\T{0}{0}{0} \mapsto \T{-i}{j}{k}$ (with the lamplighter being located to the right of the $-i$).  That is, the relation is $\lambda_i =  \nu_i  \mu_i$ for some $i \in \Z$.  So around $\partial X$ we read one of the defining relations in the presentation given in the theorem.  

Finally, we show that every edge-loop in $\mathcal{H}_2(\Z)$ which corresponds to  a defining relation bounds a $2$-cell.  So suppose $\rho: S^1 \to  \mathcal{H}_2(\Z)$, given by  $r \mapsto \rho(r) = (p_0(r), p_1(r), p_2(r))$, is a loop in the 1-skeleton of $\mathcal{H}_2(\Z)$ and around $\rho$ we read one of the defining relations.  Then for each $j$, such are the defining relations,  the image of the loop $r \mapsto p_j(r)$ is in a single edge $I_j$ of $\mathcal{T}_{\Z}$  and, by a similar analysis to that above, $$\set{ \left.  \, (u_0, u_1, u_2) \in \mathcal{T}_{\Z}^3 \   \right| \   u_j \in I_j \textup{ and } h(u_0) + h(u_1) + h(u_2)=0 \, }$$ is a 2-cell of $\mathcal{H}_2(\Z)$  with boundary circuit $\rho$. 
    
So, as no edge-loop in either $\mathcal{H}_2(\Z)$ or in the Cayley 2-complex is the boundary of two 2-cells, the result it proved.

\bibliographystyle{plain}
\bibliography{$HOME/Dropbox/Bibliographies/bibli}

\small{ 
\ni  \textsc{Margarita Amchislavska } \rule{0mm}{6mm} \\
Department of Mathematics,
Cornell University, 310 Malott Hall, Ithaca, NY 14850, USA \\ \texttt{ma569@cornell.edu}}

\small{ 
\ni  \textsc{Timothy R.\ Riley} \rule{0mm}{6mm} \\
Department of Mathematics,
Cornell University, 310 Malott Hall, Ithaca, NY 14850, USA \\ \texttt{tim.riley@math.cornell.edu}, \
\href{http://www.math.cornell.edu/~riley/}{http://www.math.cornell.edu/$\sim$riley/}}

\end{document}